\documentclass{SCAE}
\numberwithin{equation}{section}
\usepackage{amsmath,amssymb,amscd,amsthm,amsfonts}
\usepackage[all,cmtip]{xy} 
\usepackage{hyperref}

\newcommand{\A}{\mathbb{A}}
 
\newcommand{\R}{\mathbb{R}} \newcommand{\CC}{\mathbb{C}}
 \newcommand{\rO}{\mathrm{O}}
\newcommand{\cS}{\mathcal{S}}\newcommand{\cH}{\mathcal{H}}

\newcommand{\cZ}{\mathcal{Z}}

\newcommand{\half}{\frac {1} {2}}

\renewcommand{\c}[1]{\mathcal{#1}}
\newcommand{\f}[1]{\mathfrak{#1}}
\renewcommand{\r}[1]{\mathrm{#1}}

\newcommand{\til}{\widetilde}
\newcommand{\trpz}[1]{{{}^{\mathrm{t}}\negthinspace{#1}}}

\DeclareMathOperator{\rk}{\mathrm{rk}}
\DeclareMathOperator{\GL}{\mathrm{GL}}

\DeclareMathOperator{\Sym}{\mathrm{Sym}}

\DeclareMathOperator{\Sp}{\mathrm{Sp}}

\DeclareMathOperator{\tr}{\mathrm{tr}}

\DeclareMathOperator{\Ind}{\mathrm{Ind}}

\newcommand{\isom}{\cong}

\DeclareMathOperator{\rank}{\mathrm{rank}}

\DeclareMathOperator{\lmod}{\backslash}
\let\Re\undefined
\DeclareMathOperator{\Re}{\mathrm{Re}}
\newcommand{\FJ}{\mathrm{FJ}}

\newcommand{\smatrix}[4]{\ensuremath\bigl( \begin{smallmatrix}
#1&#2\\ #3&#4
\end{smallmatrix} \bigr)}
\newcommand{\form}[2]{\langle{#1},{#2}\rangle}

\usepackage[pagewise]{lineno}
\usepackage{url}
\newcommand{\REG}{\mathrm{REG}}
\newcommand{\abc}{\mathrm{abc}}

\begin{document}
\Year{2016} %
\Month{January}
\Vol{59} %
\No{1} %
\BeginPage{1} %
\EndPage{XX} %
\AuthorMark{Wu C}
\ReceivedDay{November 22, 2015}
\AcceptedDay{March 8, 2016}
\PublishedOnlineDay{; published online January 22, 2016}
\DOI{10.1007/s11425-015--0770-7} 

\title[Rallis Inner Product Formula]{On a Critical Case of Rallis Inner Product Formula}{}


\author[1]{WU Chenyan}{Corresponding author}

\address[{\rm1}]{Shanghai Center for Mathematical Sciences, Fudan University, 
  220 Handan Rd. Shanghai, {\rm 200433}, China;}
\Emails{chywu@fudan.edu.cn}\maketitle


 {\begin{center}
\parbox{14.5cm}{\begin{abstract}
  Let $\pi$ be a genuine cuspidal representation of the metaplectic
  group of rank $n$. We consider the theta lifts to the orthogonal
  group associated to a quadratic space of dimension $2n+1$. We show a 
  case of regularised Rallis inner product formula that relates the
  pairing of theta lifts to the central value of the Langlands
  $L$-function of $\pi$ twisted by a character. The bulk of this article focuses on proving a case of
  regularised Siegel-Weil formula,  on which the Rallis inner product formula is based and whose proof is missing in the liter  ature.\vspace{-3mm}
\end{abstract}}\end{center}}



\renewcommand{\baselinestretch}{1.2}
\begin{center} \renewcommand{\arraystretch}{1.5}
{\begin{tabular}{lp{0.8\textwidth}} \hline \scriptsize
{\bf Citation:}\!\!\!\!&\scriptsize Chenyan WU. \makeatletter\@titlehead.
Sci China Math, 2016, 59,
 doi:~\@DOI\makeatother\vspace{1mm}
\\
\hline
\end{tabular}}\end{center}

\baselineskip 11pt\parindent=10.8pt  \wuhao

\section{Introduction}
\label{sec:intro}

In this article we study the case of the regularised Rallis inner
product formula that relates the pairing of theta lifts to the central
value of Langlands $L$-function. We study a case of the regularised
Siegel-Weil formula first, as it is a key ingredient in the proof. It should be pointed out that even though this case of Siegel-Weil formula is stated in several places, for example, in \cite{MR2822859}, a proof for the metaplectic case is never written down and it is not clear if the proof for the symplectic case in \cite{MR1289491} or that for the unitary case in \cite{MR2064052} generalises readily. We need analytic continuation of degenerate Whittaker functionals on  certain submodules $R_n (U_v)$ of the induced representation $I (\chi_\psi\chi,s_0)$ at every local place. This is not yet known for metaplectic groups at archimedean places.  To work around it, we use some lengthy induction arguments that refine the arguments in \cite{MR1411571,MR1863861}. We take no credit in the originality of the method. This case of the Siegel-Weil formula has already been used to prove other theorems. For example, Gan, Qiu and Takeda\cite{MR3279536} used it as the basis for induction in the proof of the Siegel-Weil formula in the so-called second term range. Thus a detailed proof for this well-known case is needed.  Our proof is indebted to various work that preceded it that made clear structures and properties of various objects. 

Let $k$ be a number field and let $\A$ be its adeles. Let $U$ be a
vector space of dimension $m$ over $k$ with  quadratic form $Q$ and
let $r$ denote its Witt index. We consider the reductive dual pair
$H=\rO(U)$ and $G=\Sp(2n)$ where $n$ is the rank of the symplectic
group. We use $\til{G} (\A)$ to denote the metaplectic group which is
the nontrivial double cover of $G(\A)$. We remark that even though there is no algebraic group $\til{G}$, we prefer writing $\til{G} (\A)$ to
$\til{G (\A)}$ for aesthetic reasons.

Set $s_0=(m-n-1)/2$. We review the Siegel-Weil formula in general and
point out what is special with the case $m=n+1$. Let $\Phi$ be in the
space of Schwartz functions $\cS_0(U^n(\A))$ on which
$\til{G} (\A)\times H(\A)$ acts via the Weil representation $\omega$. To talk about Weil representation, in fact, we need to fix an additive character $\psi$ of $k\lmod \A$. Thus $\omega$ depends on $\psi$, but we have suppressed it from notation.
We associate to the Schwartz function $\Phi$ the Siegel-Weil section $f_\Phi(g,s) :=
|a(g)|^{s-s_0}\omega(g)\Phi(0)$ for $g\in \til{G} (\A)$. We refer the reader to  Sec. \ref{sec:notations} for the notation $|a(g)|$ and any other unexplained notation in this introduction. This section lies in the space
of some induced representation
$\Ind_{\til{P} (\A)}^{\til{G} (\A)}\chi\chi_\psi |\ |^s$ where $P$ is a
Siegel parabolic of $G$ and tilde means taking preimage in the
metaplectic group. Here the character $\chi\chi_\psi$ of ${\til{P} (\A)}$ is determined by $U$ and $\psi$. Form the Siegel Eisenstein series
$E(g,s,f_\Phi)$. It can be meromorphically continued to the whole
$s$-plane.  Also form the theta integral $I(g,\Phi)$. The theta
integral is not necessarily absolutely convergent. Kudla and Rallis
treated the regularisation of the theta integral in
\cite{MR1289491}. Their regularised theta
integral involves a complex variable $s$ and may have a simple or
double pole at $s= (m-r-1)/2$.  Ichino\cite{MR1863861} used an element in the Hecke algebra at a local place to regularise the theta integral
and his regularised theta integral $I_\REG(g,\Phi)$ coincides with the leading
term of Kudla-Rallis's regularised theta integral in the range $m\le
n+1$.  As it is the complementary space $U'$ of $U$ that is used in \cite{MR1863861}, in fact, the range $m=\dim U \ge n+1$ is used.   We say $U'$ is the complementary space of
$U$ if $\dim U + \dim U' = 2n+2$ and the two vector spaces are in the
same Witt tower, meaning that $U'$ is $U$ with some hyperbolic planes
stripped away or added on. Then very loosely speaking, the Siegel-Weil formula gives an identity
between the leading term of the Siegel Eisenstein series at $s_0$ and
the regularised theta integral associated to $U'$. A formulation that does not use the complementary space can be found in \cite{MR2822859}.

When the Siegel Eisenstein series and the theta integral are both
absolutely convergent, Weil\cite{MR0223373}
proved the formula in great generality.  In the case where the groups
under consideration are orthogonal group and metaplectic group, Weil's
condition for absolute convergence for the theta integral is $m-r>n+1$
or $r=0$.  The Siegel Eisenstein series $E(g,s,f_\Phi)$ is absolutely
convergent for $\Re s > (n+1)/2$. Thus if $m>2n+2$, $E(g,s,f_\Phi)$ is
absolutely convergent at $s_0=(m-n-1)/2$ and the theta integral is
also absolutely convergent.

Then assuming only the absolute convergence of theta integral, i.e.,
$m-r>n+1$ or $r=0$, Kudla and Rallis in
\cite{MR946349} and \cite{MR961164}
proved that the analytically continued Siegel Eisenstein series is
holomorphic at $s_0$ and showed that the Siegel-Weil formula holds
between the value at $s_0$ of the Siegel Eisenstein series and the
theta integral.

In \cite{MR1289491} Kudla and Rallis
introduced the regularised theta integral to remove the requirement of
absolute convergence of the theta integral. The formula then relates
the residue of Siegel Eisenstein series at $s_0$ with the leading term
of their regularised theta integral. However they worked under the
condition that $m$ is even, in which case the Weil representation
reduces to a representation of the symplectic group.  Note that here
to avoid excessive complication we have excluded the split binary case which
requires a slightly different statement. In a later
paper\cite[Thm. 3.1]{MR1491448} Kudla announced the Siegel-Weil
formula for $m=n+1$ without parity restriction, but he referred
back to \cite{MR1289491} for a proof of the
$m$ odd case. The formula also appears in Yamana's
paper\cite{MR2822859}, which gives a more transparent proof for the
regularised Siegel-Weil formula, but for the case $m=n+1$ he wrote down a proof only for the Siegel-Weil formula with the Eisenstein series on the orthogonal group. Thus for our case with Eisenstein series on the metaplectic group, there is still  no explicit proof.

For $m$ odd Ikeda in \cite{MR1411571} proved
an analogous formula.  The theta integral involved is  associated to the complementary space
$U'$. In the case $m=n+1$, $U'$ is just $U$ and
the Eisenstein series is holomorphic at $s_0=0$. Ikeda's theta
integral does not require regularisation since he assumed that the
complementary space $U'$ of $U$ is anisotropic in the case
$n+1<m<2n+2$ or that $U$ is anisotropic in the case $m=n+1$. The
method for regularising theta integral was generalised by
Ichino\cite{MR1863861} so that $k$ is no longer required
to have a real place as in
\cite{MR1289491}. Instead of using
differential operator at a real place as did
Kudla-Rallis\cite{MR1289491}, he used a Hecke
operator at a finite place. In Ichino's notation the Siegel-Weil
formula is an equation between the residue at $s_0$ of the Siegel
Eisenstein series and the regularised theta integral
$I_{\REG,Q'}(g,\pi_K\pi^{Q'}_{Q}\Phi)$ where $Q'$ is the quadratic
form of the complementary space. He considered the case where
$n+1<m\le 2n+2$ and $m-r\le n+1$ with no parity restriction on
$m$. The interesting case $m=n+1$ with $m$ odd  is still left
open. We will prove this case in Thm.~\ref{thm:Siegel-Weil}.

Next we deduce the critical case of the Rallis inner product formula from this case of the Siegel-Weil formula.
Let $G^\square$ be the `doubled group' of $G$. Thus it is a symplectic group of rank $2n$. 
Let $\pi$ be a genuine cuspidal representation of
$\til {G} (\A)$ and consider theta lifts of $\pi$ from $\til {G} (\A)$ to
$\r{O}(U)$ where $U$ has dimension $2n+1$.
Via the doubling method, to compute the inner products of such theta lifts we
ultimately need to apply the regularised Siegel-Weil formula where the orthogonal group is $\r{O}(U)$ and the metaplectic group is 
$\til { G^\square} (\A)$.  Via \cite{MR892097}, this leads to a central $L$-value.
We show that the pairing of theta lifts can be expressed in terms of  the central value of the Langlands $L$-function of $\pi$
twisted with a  character
(Thm.~\ref{thm:rallis_inner_product}).  We also show a relation
between the non-vanishing of central $L$-value and the non-vanishing of
theta lift in Thm.~\ref{thm:nonvanishing_theta_lifts}.  The arithmetic
inner product formula which involves the central derivative of 
$L$-function will be part of my future work. 

The idea of the proof of the regularised Siegel-Weil formula
originates from\cite{MR1289491}. We try to show the
identity by comparing the Fourier coefficients of the Siegel
Eisenstein series and those of the regularised theta integral. We follow closely what was done in
\cite{MR1411571,MR1863861}. In fact most of the arguments carry through.  Let $A(g,\Phi)$
denote the difference between $E(g,s,\Phi)|_{s=s_0}$ and a certain multiple of
$I_{\REG}(g,\Phi)$. Via the theory of Fourier-Jacobi
coefficients we are able to show the vanishing of nonsingular Fourier
coefficients of $A$ can be deduced from lower rank cases. However for the case $(m,r)= (3,1)$, not all Fourier-Jacobi coefficients can lead to lower rank cases. Thus we need further analysis to show that the vanishing of  nonsingular Fourier coefficients of $A$ can be deduced from consideration of those Fourier-Jacobi coefficients leading to lower rank cases. We also fill in some details omitted in \cite{MR1863861}. 
Finally this case of the Siegel-Weil formula coupled with doubling
method enables us to show the critical case of the Rallis inner product formula.

The structure of the article is organised as follows. After setting up necessary notation in Sec.~\ref{sec:notations}, we state the theorem of regularised Siegel-Weil formula in the so-called boundary case in Sec.~\ref{sec:stmt_thm} where we also summarise previous results on the regularisation of theta integrals and some properties of Siegel Eisenstein series. Then in Sec.~\ref{sec:fj-coeff} we compute  the Fourier-Jacobi coefficients for both regularised theta integral and Siegel Eisenstein series. The proof of the regularised Siegel-Weil formula is done in Sec.~\ref{sec:pf_main_thm} by using an induction process enabled by the Fourier-Jacobi coefficients. Finally in Sec.~\ref{sec:inner_product} we deduce our main result of the critical case of the Rallis inner product formula from the regularised Siegel-Weil formula. As a result, we get a relation between non-vanishing of theta lifts and non-vanishing of certain L-function at the critical point $s=1/2$. Our future research project on the arithmetic Rallis inner product formula will depend on this case of the Rallis inner product formula.

\Acknowledgements{I would like to thank my adviser, Professor Shou-Wu Zhang, for lots of
helpful conversations and encouragement during the preparation of the
bulk of this article. Thanks also go to Professor Dihua Jiang for his patience with my nagging questions on the topic and for 
pointing out a flaw in the original argument.
}

\section{Notation and Preliminaries}
\label{sec:notations}
Let $k$ be a number field and $\A$ its adele ring.  Let $U$ be a
vector space of dimension $m$ over $k$ with quadratic form $Q$. 
 The associated bilinear
form $\form{\cdot}{\cdot}_Q$ on $U$ is defined by
$\langle x, y\rangle_Q= Q(x+y)-Q(x)-Q(y)$. Thus $\form{x}{x}_Q =
2Q(x)$. The choice is made so that our notation conforms with \cite{kudla:_notes_local_theta_corres}. Let $\Delta_Q=\det\form{\cdot}{\cdot}_Q$. It is equal to $2^m\det Q$ in our
setup. Let $r$ denote the Witt index of $U$, i.e., the dimension of a
maximal isotropic subspace of $U$. Let $H= \rO(U)$ denote the
orthogonal group of $(U,Q)$. The action of $H$ on $U$ is from the left and we
view the vectors in $U$ as column vectors. Let $G= \Sp_{2n}$ be the symplectic group
of rank $n$ that fixes the symplectic form
$\smatrix{0_n}{1_n}{-1_n}{0_n}$. 
The action of $G$ on the symplectic space is from the right. For each real or finite place $v$ of $k$ there is a
unique non-trivial double cover $\til {G} (k_v)$ of $G (k_v)$ and when
$v \nmid \infty$ and $v \nmid 2$, $\til {G} (k_v)$ splits uniquely over the standard
maximal compact subgroup $K_v$ of $G (k_v)$. For each complex place $v$ of $k$, the metaplectic double cover
 of $G (k_v)$  splits. Still let $K_v$ denote
the image of the splitting.    We define $\til{G}(\A)$ to be the quotient
of the restricted product of $\til {G} (k_v)$ with respect to these
$K_v$'s by the central subgroup
\begin{equation*}
  \{z\in \oplus_v \mu_2 |
\text{evenly many components are $-1$}\}.
\end{equation*}
 This is a non-trivial
double cover of $G(\A)$.  We describe more precisely the multiplication law of $\til{G}(k_v)$. As a set $\til{G}(k_v)$ is isomorphic to $G (k_v) \times \mu_2$. Then the multiplication law on $G (k_v) \times \mu_2$ is given
by
\begin{equation*}
  (g_1,\zeta_1)(g_2,\zeta_2)=(g_1g_2,c_v(g_1,g_2)\zeta_1\zeta_2)
\end{equation*}
where $c_v(g_1,g_2)$ is Rao's $2$-cocycle on
$G(k_v)$ with values in $\mu_2$. The definition of $c_v$ can be
found in \cite[Theorem 5.3]{MR1197062} where it is denoted by $c^{\sim}$. There
the factor $(-1,-1)^{\frac{l(l+1)}{2}}$ should be
$(-1,-1)^{\frac{l(l-1)}{2}}$ as pointed out, for example, in Kudla's
notes \cite[Remark. 4.6]{kudla:_notes_local_theta_corres}.

Fix a nontrivial additive character $\psi$
of $\A/k$ and set $\psi_S(\cdot)= \psi(S\cdot)$ for $S\in
k^\times$.
The Weil representation $\omega$ is a representation of $\til{G}(\A)\times
H(\A)$ and can be realised on the space of Schwartz functions $\cS(U^n(\A))$. 
To give explicit formulae, we set for $A\in \GL_n$
\begin{equation*}
  m(A) = \smatrix{A}{0}{0}{\trpz{A}^{-1}} \in \Sp_{2n}
\end{equation*}
and for $B \in \Sym_n$
\begin{equation*}
  n(B)= \smatrix{1}{B}{0}{1} \in \Sp_{2n}
\end{equation*}
and 
\begin{equation*}
  w_n = \smatrix {} {1_n} {-1_n} {}.
\end{equation*}
Locally the Weil representation $\omega$ is
characterised by the following properties (see e.g.,
\cite[(1.1)-(1.3)]{MR1411571} ). For $\Phi\in \cS(U^n(k_v))$, $X\in U^n(k_v)$, $A\in\GL_n(k_v)$,
$B\in\Sym_n(k_v)$, $\zeta\in\mu_2$ and $h\in H(k_v)$:
\begin{align}
  \omega_v((m(A),\zeta))\Phi(X) &= \chi_v ( A)\chi_{\psi_v}( ( A, \zeta))|\det
  A|_v^{m/2}\Phi(XA),\\
  \label{eq:weil_rep_nilp}
  \omega_v((n(B),\zeta))\Phi(X) &=
  \zeta^m\psi_v(\frac{1}{2}\tr( \form{X}{X}_{Q_v}B))\Phi(X),\\
  \omega_v((
  w_n^{-1},\zeta))\Phi(X) &=
  \zeta^m\gamma_v(\psi_v\circ Q_v)^{-n}\c{F}\Phi (-X),\\
  \omega_v(h)\Phi(X)&=\Phi(h^{-1}X)
\end{align}
where the notation is explained below.

The quantity
 $\gamma_v (\psi\circ Q_v)$ denotes the Weil index of the character of second degree
$x\mapsto \psi_v\circ Q_v(x)$ and has values in $8$-th roots of unity and as a shorthand $\gamma_v (\psi)$ denotes 
the Weil index of the character of second degree
$x\mapsto \psi_v (x^2)$.
For $a\in
k_v^\times$, define also 
\begin{equation*}
 \gamma_v(a,\psi_{v,\half}) =
\frac{\gamma_v(\psi_{v,\frac{a}{2}})}{\gamma_v(\psi_{v,\half})}.  
\end{equation*}
The restriction of the double cover to the Siegel-Levi subgroup is isomorphic to  $\til{\GL}_n(k_v) \isom \GL_n(k_v) \times \mu_2$ and the group law is given by
\begin{equation*}
  (g_1,\zeta_1)\cdot (g_2,\zeta_2) = (g_1 g_2, \zeta_1\zeta_2 (\det g_1, \det g_2)_{k_v})
\end{equation*}
where $(\ ,\ )_{k_v}$ denotes the Hilbert symbol of $k_v$. Then
$\chi_{\psi_v}$ is a character of $\til{\GL} _n(k_v)$ given by 
\begin{equation}\label{eq:chi-psi}
  \chi_{\psi_v} ((A,\zeta)) =
  \begin{cases}
    \zeta \gamma_v (\det A, \psi_{v,\half})^{-1} \quad &\text {if $m$ is odd,}\\
    1 & \text{if $m$ is even.}
  \end{cases}
\end{equation}
The symbol $\chi_v$ denotes the quadratic character of $\GL_n (k_v)$ defined as follows:
\begin{equation}\label{eq:chi-U-v}
  \chi_v(A)= (\det A, (-1)^{\frac{m(m-1)}{2}} \Delta_{Q_v})_{k_v}.
\end{equation}
The matrix $\langle X,X \rangle_{Q_v}$ has $\langle X_i,X_j
\rangle_{Q_v}$ as $( i, j)$-th entry if we write $X=(X_1,\ldots,X_n)$ with
$X_i$ column vectors in $U(k_v)$. The Fourier transform $\c {F}\Phi$ of $\Phi$ with
respect to $\psi_v$ and $Q_v$ is defined to be
\begin{equation*}
  \c{F}\Phi(X)=\int_{U^n(k_v)}\psi_v(\tr\langle X,Y\rangle_{Q_v})\Phi(Y)dY.
\end{equation*}

From the Weil representation we get the theta function
\begin{equation*}
  \Theta(g,h;\Phi) = \sum_{u\in U^n(k)} \omega(g,h)\Phi(u)
\end{equation*}
where $g\in \til{G}(\A)$, $h\in H(\A)$ and $\Phi\in
\cS(U^n(\A))$. The series is absolutely convergent. Consider the theta integral
\begin{equation*}
  I(g,\Phi) = \int_{H(k)\lmod H(\A)} \Theta(g,h;\Phi) dh
\end{equation*}
which may not  converge. It is well-known that it is
absolutely convergent for all $\Phi\in \cS(U^n(\A))$ if either $r=0$
or $m-r>n+1$. Thus in the case considered in this paper we will need
to regularise the theta integral unless $Q$ is anisotropic. The regularised theta integral is a natural extension of $I$. It is defined in Sec.~\ref{sec:reg_theta} and is denoted by $I_\REG(g,\Phi)$.

Let $P$ be the Siegel parabolic subgroup of $G$ and $N$ the unipotent
radical. Since $\til { G} (k_v)$ splits uniquely over $N (k_v)$, we will
still use $N (k_v)$ to denote the image of $N (k_v)$ under the
splitting. Also $N (\A)$ denotes the image of $N (\A)$ under the
splitting of $\til {G} (\A)$ over $N (\A)$. Since $\til {G} (\A)$
splits uniquely over $G (k)$, we also identify $G (k)$ with its image
under the splitting. For any subgroup $J$ of $G (\A)$ we let $\til
{J}$ denote the preimage of $J$ under the projection $\til{G}(\A)
\rightarrow G (\A)$.

For $g\in\til{G}(\A)$ decompose $g$ as $g=m(A)nk$ with
$A\in\GL_n(\A)$, $n\in N(\A)$ and $k\in\til{K}$. Set $a(g) =\det A$
in any such decomposition of $g$. Even though there are many choices for $a (g)$,  the quantity $|a(g)|_\A$ is
well-defined.  Set $s_0=(m-n-1)/2$. The Siegel-Weil section associated to
$\Phi\in\cS(U^n(\A))$ is then defined to be
\begin{equation*}
  f_\Phi (g,s) = | a(g)|_\A^{s-s_0}\omega(g)\Phi(0).
\end{equation*}
This is a section in the normalised induced representation
$\Ind_{\til{P}(\A)}^{\til{G}(\A)} \chi\chi_\psi|\; |_\A^{s}$. We
introduce also the notion of weak Siegel-Weil section. It is a
section of the induced representation such that it agrees with a
Siegel-Weil section at $s=s_0$.

Now we define the Eisenstein series
\begin{equation*}
  E(g,s,f_\Phi)=\sum_{\gamma \in P(k)\lmod G(k)}
  f_\Phi(\gamma g,s).
\end{equation*}
It is absolutely convergent for $\Re(s)>(n+1)/2$ and has meromorphic
continuation to the whole $s$-plane when $\Phi$ is
$\til{K}$-finite. In the case where $m=n+1$, $s_0$ is  $0$ and
$E(g, s,f_\Phi)$ is known to be holomorphic at $s=s_0$.

\section{Statement of the Siegel-Weil Formula}
\label{sec:stmt_thm}
 Let $\cS_0(U^n(\A))$ denote the
$\til{K}$-finite part of $\cS(U^n(\A))$.  As our main concern is for
$U$ with odd dimension, we exclude the split binary case which
requires a separate treatment in the statement. We normalise the Haar measure on
$H(k)\lmod H(\A)$ so that it has volume $1$. We will show the
following
\begin{theorem}\label{thm:Siegel-Weil}\quad
  Assume that $m=n+1$ and that $U$ is not  split binary. Then
  \begin{equation*}
    E(g,s,f_\Phi)|_{s=0}= 2I_{\REG}(g,\Phi)
  \end{equation*}
  for all $\Phi\in\cS_0(U^n(\A))$.
\end{theorem}
\begin{remark}\quad
  The regularised theta integral $I_{\REG}$ will be defined in Sec.~
  \ref{sec:reg_theta}. For the split binary case, the Siegel-Weil
  formula also relates the leading term of the Eisenstein series and the regularised theta integral, but as the Eisenstein series vanishes at $s=0$, the equation takes a different form. As m is even, it is
  covered by Kudla-Rallis's
  result\cite{MR1289491}. A more precise statement for the split binary case is  in \cite [Proposition 5.8.(ii)] {MR2822859}:
  \begin{equation*}
    \frac{\partial}{\partial s}E(g,s,f_\Phi)|_{s=0}= 2I_{\REG}(g,\Phi).
  \end{equation*}
\end{remark}

\subsection{Regularisation of Theta Integral}
\label{sec:reg_theta}

When $Q$ is isotropic and $m-r \le n+1$, the theta integral is not
necessarily absolutely convergent for all $\Phi\in\cS(U^n(\A))$. We summarise how one
regularises theta integrals according to Ichino\cite[Section
1]{MR1863861}.  We make  use of the notation there and
explain what they are without going into too much detail.

Take $v$ a finite place of $k$. If $v \nmid 2$ then
there is a canonical splitting of $\til{G} (k_v)$ over $K_{G,v}$, the standard
maximal compact subgroup of $G (k_v)$. Identify $K_{G,v}$ with the image of the
splitting. Let $\cH_{G,v}$ and $\cH_{H,v}$ denote the spherical Hecke algebras
of $\til{G} (k_v)$ and $H (k_v)$:
\begin{align*}
  \cH_{G,v} &= \{ \alpha\in \cH(\til{G} (k_v) //K_{ G,v}) | \alpha(\epsilon
  g)=\epsilon^m\alpha(g) \text{ for all } g\in \til{G} (k_v) \},\\
  \cH_{ H,v} &= \cH(H (k_v) //K_{H,v})
\end{align*}
where $\epsilon = (\mathbf{1}_{2n},-1) \in \til{G} (k_v)$.

The following is due to Ichino\cite[Section
1]{MR1863861}.
\begin{proposition}\quad
  Assume $m\le n+1$ and $r\neq 0$. Fix $\Phi\in \cS(U^n(\A))$ and
  choose a `good' place $v$ of $k$ for $\Phi$. Then there exists a Hecke
  operator $\alpha\in\cH_{G,v}$ satisfying the following conditions:
  \begin{enumerate}
  \item $I(g,\omega(\alpha)\Phi)$ is absolutely convergent for all
    $g\in\til{G}(\A)$;
  \item $\theta(\alpha).\mathbf{1} = c_\alpha . \mathbf{1}$ with
    $c_\alpha\neq 0$.
  \end{enumerate}
\end{proposition}
\begin{remark}\quad
  For the notion of good place please refer to \cite[Page
  209]{MR1863861}. In the above proposition, $\theta$ is the algebra
  homomorphism from $\cH_{G,v}$ to $\cH_{H,v}$
  such that $\omega(\alpha)=\omega(\theta(\alpha))$ as in
  \cite[Prop 1.1]{MR1863861}.  The trivial representation
  of $H (k_v)$ is denoted by $\mathbf{1}$ here.
\end{remark}
\begin{definition}\quad
  Define the regularised theta integral by
  \begin{equation*}
    I_{\REG}(g,\Phi)=c_\alpha^{-1}I(g,\omega(\alpha)\Phi).
  \end{equation*}
\end{definition}
\begin{remark}\quad
  The above definition is in fact independent of the choice of $v$ and
  $\alpha$.  To unify notation also write $I_{\REG}(g,\Phi)$ for
  $I(g,\Phi)$ when $Q$ anisotropic.
\end{remark}

It should be pointed out that $I_\REG$ is a natural extension of $I$
in the following sense. Let $\cS(U^n(\A))_\abc$ denote the (nonzero)
subspace of $\cS(U^n(\A))$ consisting of all $\Phi$'s such that
$I(g,\Phi)$ is absolutely convergent for all $g$. Then $I$ defines an
$H(\A)$-invariant map
\begin{equation*}
  I: \cS(U^n(\A))_\abc \rightarrow \c{A}^\infty(G)
\end{equation*}
where $\c{A}^\infty(G)$ is the space of smooth automorphic forms on
$\til{G}(\A)$ (left-invariant by $G(k)$) without the
$\til{K}_G$-finiteness condition.

\begin{proposition}\quad \emph{\cite[Lemma 1.9]{MR1863861}}
  Assume $m\le n+1$. Then $I_{\REG}$ is the unique $H(\A)$-invariant
  extension of $I$ to $\cS(U^n(\A))$.
\end{proposition}

\subsection{Siegel Eisenstein Series}
\label{sec:siegel_eisenstein}
Now we define the Siegel Eisenstein series.  Temporarily let
$\til{\chi}$ be an arbitrary character of $\til{P}(\A)$. Let
$I(\til{\chi}, s)$ denote the induced representation
$\Ind^{\til{G}(\A)}_{\til{P}(\A)} \til{\chi} |\ |^s$. Then the holomorphic
sections of $I(\til{\chi}, s)$ are functions $f$ such that
\begin{enumerate}
\item $f(g,s)$ is holomorphic with respect to $s$ for each
  $g\in\til{G}(\A)$,
\item $f(pg,s)=\til{\chi}(p) |a(p)|_\A^{s+(n+1)/2} f(g,s)$ for
  $p\in\til{P}(\A)$ and $g\in\til{G}(\A)$ and
\item $f(\cdot,s)$ is $\til{K}_G$-finite.
\end{enumerate}
For $f$ a holomorphic section of $I(\til{\chi}, s)$ we form the Siegel
Eisenstein series
\begin{equation*}
  E(g,s,f)=\sum_{\gamma\in P(k)\lmod G(k)} f(\gamma g,s).
\end{equation*}
Note that $\til{G}(\A)$ splits over $G(k)$ and thus $G (k)$ is
identified with a subgroup of $\til {G} (\A)$.

We will specialise to the case where $\til{\chi}$ is the trivial
extension from $\til {M} (\A)$ to $\til {P} (\A)$ of the character $\chi_U\chi_\psi$
whose local versions, $\chi_{U,v}$ and $\chi_{\psi_v}$,  are defined in
\eqref{eq:chi-U-v} and  \eqref{eq:chi-psi} respectively.  For $\Phi\in \cS_0(U^n(\A))$ define the Siegel-Weil section by
\begin{equation*}
  f_\Phi(g,s)=|a(g)|^{s-\frac{m-n-1}{2}}\omega(g)\Phi(0).
\end{equation*}
Then $f_\Phi$ is a holomorphic section of $I(\chi_U\chi_\psi,s)$. The
Eisenstein series $E(g,s,f_\Phi)$ is absolutely convergent for
$\Re(s)>(n+1)/2$ and has meromorphic continuation to the whole
$s$-plane. We know from the summary given in\cite[Page
216]{MR1863861} that if $m=n+1$, $E(g,s,f_\Phi)$ is
holomorphic at $s=0$.

The following definition will be useful later.
\begin{definition}\quad
  \label{defn:weak_SW_section}
  A holomorphic section $f\in I(\chi_U\chi_\psi,s)$ is said to be a weak Siegel-Weil
  section associated to $\Phi\in \cS(U^n(\A))$ if
  $f(g,\frac{m-n-1}{2})=\omega(g)\Phi(0)$.
\end{definition}

Now we discuss some local aspects. Define similarly $I_v(\chi_{U,v} \chi_{\psi_v},
s)$ in the local cases.  Fix one place $v$. For every $w\neq v$, fix
$\Phi_w^0\in\cS(U^n(k_w))$ and let $f_w^0(g_w,s)$ be the associated
Siegel-Weil section in $I_v(\chi_{U,v} \chi_{\psi_v}, s)$ where we suppress the
subscript $\Phi_w^0$ to avoid clutter.  Then if $m=n+1$ we have a map
\begin{equation}
  \begin{split}
    I_v(\chi_{U,v} \chi_{\psi_v},s) &\rightarrow \c{A}(G)\\
    f_v &\mapsto E(g,s,f_v \otimes (\otimes_{w\neq v} f_w^0))|_{s=0}.
  \end{split}
\end{equation}
It follows from verbatim from the proof of \cite[Prop. 2.2]{MR961164} that it is
$\til{G} (k_v)$-intertwining if $v$ is finite or
$(\f{g}_v,\til{K}_{G,v})$-intertwining if $v$ is archimedean.

\section{Fourier-Jacobi Coefficients}
\label{sec:fj-coeff}

A key step in the proof of Siegel-Weil formula is the comparison of the nonsingular
Fourier coefficients of the Eisenstein series and those of the regularised
theta integral. This comparison is done recursively by applying Fourier-Jacobi coefficients.

Let $B$ be a nonsingular matrix in $\Sym_n (k)$. It is easy to see that the $B$-th Fourier coefficient
of the Eisenstein series is a product of entire degenerate Whittaker
functions when $m$ is even. This is also known when $n=1$ and $m$ arbitrary. 
 However in the case $m$ odd and $n\neq 1$, the analytic continuation of the degenerate Whittaker distributions in the archimedean places has not been proved. To work
around the problem Ikeda\cite{MR1411571} used
Fourier-Jacobi coefficients to initiate an induction process. The vanishing of
$B$-th Fourier coefficients can be determined from Fourier coefficients of  lower dimensional objects.

We generalise slightly  the calculation done in
\cite{MR1411571} and fill in some computation omitted in \cite{MR1863861}.  First we introduce some
subgroups of $G$, describe Weil representation of a certain Heisenberg
group $V$ and then define the Fourier-Jacobi coefficients of an
automorphic form. The exposition closely follows that in
\cite{MR1863861}.

Put
\begin{align*}
  V&=\left\{ v(x,y,z)=\left(
      \begin{array}{cc|cc}
        1&x&z&y\\
        &1_{n-1}&\trpz{y}& \\ \hline
        &&1&\\
        &&-\trpz{x}&1_{n-1}
      \end{array}\right)
    \middle| x,y \in k^{n-1},z\in k \right\}, \\
  Z&=\left\{v(0,0,z)\in V\right\},\\
  L&=\left\{v(x,0,0)\in V\right\},\\
L^*&=\left\{v(0,y,0)\in V\right\},\\
  G_1&=\left\{ \left(\begin{array}{cc|cc}
        1 & & &\\
        &a & &b\\ \hline
        &&1&\\
        &c&&d
      \end{array}\right)
    \middle|
    \begin{pmatrix}
      a&b\\c&d
    \end{pmatrix}\in \Sp_{n-1}\right\},\\
  N_1&=\left\{ \left(\begin{array}{cc|cc}
        1 & & &\\
        &1_{n-1}&&n_1\\ \hline
        &&1&\\
        &&&1_{n-1}
      \end{array}\right)
    \middle| n_1\in\Sym_{n-1} \right\}.
\end{align*}
We will just
identify $L$ (resp. $L^*$) with row vectors of length $n-1$. Then $V$ is the Heisenberg group $\cH (W) = W \oplus k$ associated to the symplectic space $W = L\oplus L^*$. Note that $v (x,y,z)$ corresponds to the element $((x,y),z/2)$ in $\cH (W)$ and that $L$ and $L^*$ are transversal maximal isotropic subspaces of $W$. The symplectic form on $W$ will be denoted by $\form {\ } {\ }$ and it is given by
\begin{equation*}
  \form {v (x,0,0)} {v (0,y,0) } = x \trpz {y}.
\end{equation*}
 Temporarily we let $\psi$ be an arbitrary nontrivial character
of $k\lmod \A$. The Schr\"odinger representation $\omega$ of $V(\A)$
with central character $\psi$ can be realised on the Schwartz space
$\cS(L(\A))$ and is given by\cite[Lemma
2.2, Chap. I]{kudla:_notes_local_theta_corres}:
\begin{equation}\label{eq:heisenberg_action}
  \omega(v(x,y,z))\phi(t)=\phi(t+x)\psi(\half z+ \form{t}{y}_W+\form{x}{y}_W)
\end{equation}
for $\phi\in \cS(L(\A))$.  By the Stone-von Neumann
theorem, $\omega$ is irreducible and unique up to isomorphism. As
$G_1(\A)$ acts on $V(\A)$, the  representation $\omega$ of
$V(\A)$ can be extended to a representation $\omega$ of
$V(\A)\rtimes \til{G }_1(\A)$ on $\cS(L(\A))$. See also
\cite[Prop. 2.3]{kudla:_notes_local_theta_corres}.  Let $\til{K}_{G_1}$
denote the standard maximal compact subgroup of $\til{G }_1(\A)$ and
$\cS_0(L(\A))$ the $\til{K}_{G_1}$-finite elements in $\cS(L(\A))$.

For each $\phi\in\cS(L(\A))$ define the theta function
\begin{equation*}
  \vartheta(vg_1,\phi)=\sum_{t\in L(k)} \omega(vg_1)\phi(t)
\end{equation*}
for $v\in V(\A)$ and $g_1\in \til{G}_1(\A)$.  For $A$ an automorphic
form on $\til{G}(\A)$ and $\phi \in \cS(L(\A))$, we define the Fourier-Jacobi coefficient of $A$,
which is a function on $G_1(k)\lmod \til{G}_1(\A)$, by
\begin{equation*}
  \FJ^\phi(g_1;A)=\int_{[V]} A(vg_1)\overline{\vartheta(vg_1,\phi)}dv.
\end{equation*}
Then for $\beta\in \Sym_{n-1}(k)$, we consider the 
$\beta$-th Fourier coefficient $\FJ_{\beta}^\phi(g_1;A)$ of $\FJ^\phi(g_1;A)$.

 Let $\psi$
be our fixed additive character of $k \lmod \A$ again. With this setup
we will use the character $\psi_{S}$ as the central character in the
representation of the Heisenberg group discussed above and add subscripts in the Weil representation $\omega_S$ and Fourier-Jacobi coefficients $\FJ_S^\phi$ to indicate the additive character involved.

As we follow a different convention of the relation between quadratic
form and bilinear form from that of \cite{MR1863861}, to avoid confusion we
state and prove the following lemma.
\begin{lemma}\emph{\cite[Lemma 4.1]{MR1863861}}\quad
  \label{lemma:fj2fourier}
  Let $S\in k^\times$ and $\beta\in \Sym_{n-1}(k)$. Let $A$ be an
  automorphic form on $\til{G}(\A)$, and assume that
  $\FJ_{S,\beta}^\phi(g_1;\rho(f)A)=0$ for all $\phi\in \cS_0(L(\A))$
  and all $f$ in the Hecke algebra $\cH(\til{G}(\A))$. Then $A_B=0$ where
  \begin{equation*}
    B=
    \begin{pmatrix}
      S/2 & \\ & \beta
    \end{pmatrix}.
  \end{equation*}
\end{lemma}
\begin{proof}
  For $u\in N$ we set $b(u)$ to be the upper-right block of $u$ of
  size $n\times n$ and set $b_1(u)$ to be the lower-right block of
  $b(u)$ of size $(n-1)\times (n-1)$.  For $u_1 \in N_1$, also set $b_1 (u_1)$ to be the upper-right block of $u$ of
  size $( n-1) \times ( n-1)$. In the computation below the subscript $S$ in $\FJ_S^\phi$ is suppressed. We compute
  \begin{align*}
    &\quad     \FJ^\phi_\beta(g_1,A)\\
    &= \int_{[N_1]}\int_{[V]}    A(vu_1g_1)
    \overline{\vartheta(vu_1g_1,\phi)} \psi(-\tr(b_1(u_1)\beta))dv du_1\\
    &= \int_{[L]}\int_{[N]} A(uxg_1)
    \overline{\vartheta(uxg_1,\phi)}\psi(-\tr(b_1(u)\beta))du dx \\
    &=\int_{[L]}\int_{[N]} \sum_{t\in
      L(k)}
    A(uxg_1)\overline{\omega_S(tuxg_1)\phi(0)}\psi(-\tr(b_1(u)\beta)) du
    dx\\
    &=\int_{[L]}\sum_{t\in
      L(k)}\int_{[N]} 
    A(tuxg_1)\overline{\omega_S(tuxg_1)\phi(0)}\psi(-\tr(b_1(u)\beta)) du
    dx\\
    &=\int_{[L]}\sum_{t\in
      L(k)}\int_{[N]}
    A(utxg_1)\overline{\omega_S(utxg_1)\phi(0)}
    \psi(-\tr(b_1(u)\beta))du
    dx\\
    &=\int_{L(\A)}\int_{[N]}
    A(uxg_1)\overline{\omega_S(uxg_1)\phi(0)} \psi(-\tr(b_1(u)\beta))du
    dx.
\end{align*}
If $u$ is of the form
\begin{equation*}
  \left(
      \begin{array}{c|c}
        1_n &
        \begin{array}{cc}
          z & y\\ \trpz {y} &w
        \end{array}\\ \hline
        & 1_n
      \end{array}\right)
\end{equation*}
then it follows from \eqref{eq:heisenberg_action} and \eqref{eq:weil_rep_nilp}, $\omega_S(u)\phi(0) = \psi_{S}(\half z)\phi (0)$. Thus we find that $\FJ^\phi_\beta(g_1,A)$ is equal to
\begin{align*}
  & \int_{L(\A)}\int_{[N]}
    A(uxg_1)\overline{\omega_S(g_1)\phi(x)\psi_{S}(\half z)}\psi(-\tr(b_1(u)\beta))du
    dx\\
    =&\int_{L(\A)}\int_{[N]}
    A(uxg_1)\overline{\omega_S(g_1)\phi(x)}\psi(-\tr(b(u)B))du
    dx\\
    =&\int_{L(\A)}A_B(xg_1)\overline{\omega(g_1)\phi(x)}dx.
\end{align*}

Since $\FJ^\phi_\beta(g_1,A)=0$ for all $g_1\in \til{G}_1(\A)$ we
conclude that $A_B(g_1)=0$ for all $g_1\in \til{G}_1(\A)$. Then we
apply a sequence of $f_i \in \cH(\til{G}(\A))$ that converges to the
Dirac delta at $g\in \til{G}(\A)$ to conclude that $A_B(g)=0$ for
all $g\in \til{G}(\A)$.
 \end{proof}

\subsection{Fourier-Jacobi coefficients of the regularised theta
  integrals}
\label{sec:FJ_theta_integral}

The Fourier-Jacobi coefficient of the regularised
theta integrals is, by definition, given by:
\begin{equation*}
  \FJ_S^\phi(g_1;I_{\REG}(\Phi))= c_\alpha^{-1}\int_{[V]} \int_{[H]}
  \Theta(vg_1,h;\omega(\alpha)\Phi)\overline{\vartheta_S( v g_1,\phi)}dhdv.
\end{equation*}
In the computation below, it will be related to a regularised theta integral associated to a smaller orthogonal group. We set up some notation to describe the result of the computation.

Assume that $Q$ represents $S/2$. Decompose $U$ into an orthogonal sum $k\oplus U_1$ such that the
bilinear form $\form{\;}{\;}_Q$ is equal to
$\form{\;}{\;}_{S/2}\oplus\form{\;}{\;}_{Q_1}$ where $S/2$ and $Q_1$ are
quadratic forms on $k$ and $U_1$ respectively. In fact $S/2$ is just a
scalar in $k^\times$ and so $\form{x}{y}_{S/2}=Sxy$. Let $H_1=\rO(U_1)$.
Put
\begin{equation*}
  \Psi(u,\Phi,\phi)=\int_{L(\A)} \Phi
  \begin{pmatrix}
    1 & x\\ 0 & u
  \end{pmatrix}\overline{\phi(x)}dx
\end{equation*}
for $u\in U_1^{n-1}(\A)$ and one can check that this is a Schwartz
function in $u$. It defines a
$\til{G}_1(\A)$-intertwining map:
\begin{align*}
  \cS(U^n(\A))\otimes \cS(L(\A))& \rightarrow \cS(U_1^{n-1}(\A))\\
  \Phi \otimes \phi &\mapsto \Psi(\Phi,\phi)
\end{align*}
i.e.,
\begin{equation*}
  \omega_{\psi,Q_1} (g_1)\Psi(\Phi,\phi)=\Psi(\omega_{\psi,Q} (g_1)\Phi,\omega_{\psi_S} (g_1)\phi)
\end{equation*}
for $g_1\in \til{G}_1(\A)$. We indicate in the subscripts what additive characters and what quadratic spaces are used in the Weil representations. When there is no confusion, the subscripts will be dropped. Usually we will write $\omega_S$ for $\omega_{\psi_S}$. Notice that on $\cS(U^n(\A))$ and
$\cS(U_1^{n-1}(\A))$ the Weil representations are associated with the
character $\psi$ and on $\cS(L(\A))$ the Weil representation is
associated with the character $\psi_{S}$. Then we have:

\begin{proposition}\quad
  \label{prop:fj_theta} Assume $m=n+1$.
  Suppose that $\beta\in \Sym_{n-1}(k)$ with $\det(\beta)\neq 0$. Then
  \begin{equation*}
    \FJ^\phi_{S,\beta}(g_1; I_{\REG}(\Phi)) =\int_{H_1(\A)\lmod H(\A)} I_{\REG,\beta}(g_1,
    \Psi(\omega(h)\Phi,\phi)) dh
  \end{equation*}
if $Q$ represents $S/2$. The integral on the right-hand side is  absolutely convergent. If $Q$ does not represent $S/2$, $\FJ^\phi_{S,\beta}(g_1; I_{\REG}(\Phi))$ vanishes. 
\end{proposition}
\begin{proof}
  By the definition of Fourier-Jacobi coefficients, we need to compute
  the following integral:
  \begin{equation}\label{eq:FJ_IREG_beta}
    \begin{split}
      & \FJ^\phi_{S,\beta}(g_1; I_{\REG}(\Phi))\\
      =& c_\alpha^{-1}\int_{[N_1]}\int_{[V]}
      \int_{[H]}
      \theta(vn_1g_1,h_0,\omega(\alpha)\Phi)\overline{\vartheta_S(vn_1g_1,\phi)}\psi(-\tr b_1(n_1)\beta)dh_0 dv dn_1.
    \end{split}
  \end{equation}
  Here $b_1(n_1)$ is the upper-right block of $n_1$.

  Formally we exchange order of integration so that we integrate over
  $v$ first. This will be justified later.  Temporarily
  absorb $\omega(\alpha)$ into $\Phi$. We consider
  \begin{align}\label{eq:int_v}
    \int_{[V]}\theta(vg_1,h_0,\Phi)\overline{\vartheta_S(vg_1,\phi)}dv.
  \end{align}
  We  expand out the action of $v\in V (\A)$ and observe that
  some terms must vanish.

  Suppose $v=v(x,0,0)v(0,y,z)$. Then
  \begin{align*}
    &  \theta(vg_1,h_0,\Phi) \\
    =&\sum_{t\in U^n(k)} \omega(vg_1,h_0)\Phi(t)\\
    =&\sum_{t\in U^n(k)} \omega(v(0,y,z)g_1,h_0)\Phi(t
    \begin{pmatrix}
      1&x\\ & 1
    \end{pmatrix})\\
    =&\sum_{t\in U^n(k)} \omega(g_1,h_0)\Phi(t
    \begin{pmatrix}
      1&x\\ & 1
    \end{pmatrix})\psi(\frac{1}{2}\tr \left(\form{t
        \begin{pmatrix}
          1&x\\ & 1
        \end{pmatrix}
      }{t
        \begin{pmatrix}
          1&x\\ & 1
        \end{pmatrix}}_Q
      \begin{pmatrix}
        z&y\\ \trpz{y}&
      \end{pmatrix}\right))\\
    =&\sum_{t=\left(
        \begin{smallmatrix}
          t_1 & t_2\\ t_3 & t_4
        \end{smallmatrix}\right)}
    \omega(g_1,h_0)\Phi(t\begin{pmatrix} 1&x\\ & 1
    \end{pmatrix}) \psi(\frac{1}{2}\form{
      \begin{pmatrix}
        t_1\\t_3
      \end{pmatrix}}{\begin{pmatrix} t_1\\t_3
      \end{pmatrix}}_{Q}(z+2x\trpz{y}))
    \psi(\form{
      \begin{pmatrix}
        t_1\\ t_3
      \end{pmatrix}
    }{
      \begin{pmatrix}
        t_2\\t_4
      \end{pmatrix}
    }_Q\trpz{y})
  \end{align*}
  where $t_1\in k$, $t_2\in k^{n-1}$, $t_3\in U_1(k) $ and $t_4\in
  U_1^{n-1}(k)$. Also we expand
  \begin{equation*}
      \vartheta_S(vg_1,\phi)
    = \sum_{t\in
      L(k)}\omega_S(g_1)\phi(t+x)\psi_{S}(\half z+ \form{x}{y}+ \form{t}{y}).
  \end{equation*}
  Thus if we integrate against $z$, the integral \eqref{eq:int_v}
  vanishes unless
  \begin{equation*}
    \form{
      \begin{pmatrix}
        t_1\\t_3
      \end{pmatrix}}{\begin{pmatrix} t_1\\t_3
      \end{pmatrix}}_{Q}=S.
  \end{equation*}
We continue the computation assuming that $Q$ represents $S/2$.
  By Witt's theorem there exists some $h\in H(k)$ such that
  \begin{equation*}
    \begin{pmatrix}
      t_1\\t_3
    \end{pmatrix}=h^{-1}
    \begin{pmatrix}
      1\\0
    \end{pmatrix},
  \end{equation*}
because by our decomposition of $U$, 
\begin{equation*}
  \form{\begin{pmatrix}
      1\\0
    \end{pmatrix}}{\begin{pmatrix}
      1\\0
    \end{pmatrix}}_Q = S.
\end{equation*}
  Then the stabiliser of $\left(
    \begin{smallmatrix}
      1\\0
    \end{smallmatrix}
  \right)$ in $H(k)$ is $H_1(k)$. After changing $\left(
    \begin{smallmatrix}
      t_2\\t_4
    \end{smallmatrix}\right)
  $ to $h^{-1}\left(
    \begin{smallmatrix}
      t_2\\t_4
    \end{smallmatrix}\right)$ we find that \eqref{eq:int_v} is equal
  to
  \begin{align*}
    &\int_{[L^*]}\int_{[L]} \sum_{h\in H_1(k)\lmod H(k)}\sum_{ t_2 ,
      t_4 }\sum_{t\in L(k)}
    \omega(g_1,1_H)\Phi(h_0^{-1}h^{-1}
    \begin{pmatrix}
      1 & t_2\\ 0 & t_4
    \end{pmatrix}
    \begin{pmatrix}
      1&x\\ & 1
    \end{pmatrix})
    \overline{\omega_S(g_1)\phi(t+x)} \\
    &\quad\times\psi(\form{
      \begin{pmatrix}
        1\\ 0
      \end{pmatrix}
    }{
      \begin{pmatrix}
        t_2\\t_4
      \end{pmatrix}
    }_Q\trpz{y})\psi_{S}(-\form{t}{y}) dxdy  \\
    =&\int_{[L^*]}\int_{[L]} \sum_{h\in H_1(k)\lmod H(k)}\sum_{ t_2 ,
      t_4 }\sum_{t\in L(k)} \omega(g_1,1_H)\Phi(h_0^{-1}h^{-1}
    \begin{pmatrix}
      1 & t_2\\ 0 & t_4
    \end{pmatrix}
    \begin{pmatrix}
      1&x\\ & 1
    \end{pmatrix})\\
    &\quad\times\overline{\omega_S(g_1)\phi(t+x)}
    \psi(St_2\trpz{y})\psi_{S}(-\form{t}{y}) dxdy.
  \end{align*}
  Now the integration against $y$ vanishes unless $t=t_2$ and we find that
  the above is equal to
  \begin{align}
    & \int_{ [L]} \sum_{h\in H_1(k)\lmod H(k)} \sum_{ t_4
    } \sum_{t\in L (k)}\omega(g_1,1_H)\Phi(h_0^{-1}h^{-1}
    \begin{pmatrix}
      1 & t\\ 0 & t_4
    \end{pmatrix}
    \begin{pmatrix}
      1&x\\ & 1
    \end{pmatrix})
    \overline{\omega(g_1,1_{H_1})\phi(t+x)}dx  \nonumber\\ 
\label{eq:FJ_IREG_beta-2}
    &= \sum_{h\in H_1(k)\lmod H(k)} \sum_{ t_4\in U_1^{n-1}(k)}
    \omega(g_1,1_{H_1})\Psi(t_4,\omega(1_{\til {G} (\A)},hh_0)\Phi,\phi).
  \end{align}

  Then we consider the integration over $[N_1]$ in
  \eqref{eq:FJ_IREG_beta}. This will kill those terms in \eqref{eq:FJ_IREG_beta-2} such that
  $\form{t_4}{t_4}_{Q_1}\neq 2\beta$. Unabsorbing $\omega (\alpha)$ from $\Phi$, we find that \eqref{eq:FJ_IREG_beta} is
  equal to
  \begin{align}\label{eq:FJ_midstep}
    c_\alpha^{-1}\int_{[H]}\sum_{h\in H_1(k)\lmod
      H(k)}\sum_{\substack{t\in U_1^{n-1}(k),\\ \form{t}{t}_{Q_1}=
        2\beta} }
    \omega(g_1,1_{H_1})\Psi(t,\omega(1_{\til {G} (\A)},hh_0)\omega(\alpha)\Phi,\phi) dh_0.
  \end{align}

  We assume that $\form{\cdot}{\cdot}_{Q_1}$ represents $2\beta$,
  since otherwise the two expressions in the statement of the lemma
  are both zero and the lemma holds trivially.  As $\rk \beta = n-1$,
  $\Omega_{\beta}=\{t\in U_1^{n-1}(k) | \form{t}{t}_{Q_1}=2\beta \}$ is a single
  $H_1(k)$-orbit. Fix a representative $t_0$ of this orbit. Since the
  stabiliser of $t_0$ in $H_1(k)$ is of order $2$,
  \eqref{eq:FJ_midstep} is equal to
  \begin{equation*}
    \half c_\alpha^{-1}  \int_{[H]} \sum_{h \in H(k)}
    \omega(g_1,1_{H_1}) \Psi(t_0,\omega(1_{\til {G} (\A)},h  h_0)\omega(\alpha)\Phi,\phi) dh_0.
  \end{equation*}
  It can be checked that an analogue of the convergence lemma 
  \cite [Lemma 4.3] {MR1863861} also holds for our case and this is recorded
  here as Lemma~\ref{lemma:convergence} which shows that the above
  integral is absolutely convergent. Thus we have justified exchanging
  the order of integration in
  $\FJ_{S,\beta}^\phi(g_1;I_{\REG}(\Phi))$. We now continue the
  computation from \eqref{eq:FJ_midstep} and find that \eqref{eq:FJ_IREG_beta} is equal to
  \begin{align*}
    &    c_\alpha^{-1}\int_{H_1 (k) \lmod H (\A)}\sum_{ \form{t}{t}_{Q_1}=
      2\beta }
    \omega(g_1,1_{H_1})\Psi(t,\omega(1_{\til {G} (\A)},h)\omega(\alpha)\Phi,\phi) dh \\
=&c_\alpha^{-1}\int_{H_1 (\A) \lmod H (\A)} \int_{[H_1]}\sum_{ \form{t}{t}_{Q_1}=
      2\beta }
    \omega(g_1,h')\Psi(t,\omega(1_{\til {G} (\A)},h)\omega(\alpha)\Phi,\phi) dh' dh \\
=& \int_{H_1 (\A) \lmod H (\A)} \int_{[H_1]}\sum_{ \form{t}{t}_{Q_1}=
      2\beta }
    \omega(g_1,h')\Psi(t,\omega(1_{\til {G} (\A)}, h)\Phi,\phi) dh' dh \\
&=\int_{H_1(\A)\lmod H(\A)}
    I_{\REG,\beta}(g_1,\Psi(\omega(1_{\til {G} (\A)},h)\Phi,\phi))dh.
  \end{align*}
\end{proof}

\begin{lemma}\quad
  \label{lemma:convergence}
  Assume $m=n+1$.
  \begin{enumerate} 
  \item Let $t\in U^n(k)$. If $\rk t = n$ then $\int_{H(\A)}
    \omega(h)\Phi(t)dh$ is absolutely convergent for any $\Phi\in
    \cS(U^n(\A))$.
  \item Let $t_1\in U_1^{n-1}(k)$. If $\rk t_1=n-1$ then
    \begin{equation*}
      \int_{H(\A)} \Psi(\omega(h)\Phi,\phi,t_1)dh =
      \int_{H(\A)}\int_{L(\A)}\omega(h) \Phi
      \begin{pmatrix}
        1&x\\0&t_1
      \end{pmatrix}\overline{\phi(x)}dxdh
    \end{equation*}
    is absolutely convergent for any $\Phi\in\cS(U^n(\A))$ and
    $\phi\in\cS(L(\A))$.
  \end{enumerate}
\end{lemma}
\begin{proof}
  The argument in \cite[pp. 59-60]{MR1289491}
  also includes the case $m=n+1$ and it proves (1). For (2) consider
  the function on $U^n(\A)$
  \begin{equation*}
    \varphi(u)=\int_{L(\A)}\Phi(u
    \begin{pmatrix}
      1& x\\ & 1_{n-1}
    \end{pmatrix})\overline{\phi(x)}dx.
  \end{equation*}
  This integral is absolutely convergent and defines a smooth function
  on $U^n(\A)$. Furthermore it can be checked that
  $\varphi\in\cS(U^n(\A))$. Then we apply (1) to get (2).
\end{proof}

\subsection{Fourier-Jacobi coefficients of the Siegel Eisenstein
  series}
\label{sec:FJ_eisenstein}

In this subsection we compute the Fourier-Jacobi coefficients of the Siegel
Eisenstein series
\begin{equation*}
  \FJ_S^\phi(g_1, E(f,s))=\int_{V(k)\lmod V(\A)}E(vg_1,f,s)\overline{\vartheta_S(vg_1,\phi)}dv.
\end{equation*}
We will show a relation between the Fourier-Jacobi coefficient of the
Eisenstein series and an Eisenstein series associated to groups of lower rank. This is parallel to Prop.~\ref{prop:fj_theta}.  Let $\chi_1$ be the
character of $\GL_{n-1}(\A)$ defined analogously to $\chi$ in
\eqref{eq:chi-U-v} except that it is associated to $Q_1$
instead of $Q$. Recall that $Q_1$ is a certain quadratic form given at the beginning of Sec.~\ref{sec:FJ_theta_integral}. We may view the character $\chi_\psi$ as a character of $\til {\GL}_{n-1} (\A)$ as well. Most of  the computation here is due to Ikeda
\cite{MR1411571}. We show that
Lemma~\ref{lemma:SW_n=1} holds for more cases than was shown by Ikeda
and this is key for the case where $(m,r)=(3,1)$ and $n=2$.

\begin{proposition}\label{prop:FJ-R}\quad
  \begin{enumerate}
  \item For $\phi\in \cS_0(L(\A))$ we have
  \begin{equation*}
    \FJ_S^\phi(g_1, E(f,s)) = \sum_{\gamma\in P_1(k)\lmod G_1(k)}
    R_S(\gamma g_1, f, s,\phi)
  \end{equation*}
  where for $\Re s  >>0$
  \begin{equation}\label{eq:R}
    R_S(g_1,f,s,\phi) = \int_{L(\A)} \int_\A
    f(w_nv(0,y,z)w_{n-1}g_1,s)\overline{\omega_S(g_1)\phi(-y)\psi_{S}(\half z)}dz dy
  \end{equation} 
  is a holomorphic section of
  $\Ind_{P_1(\A)}^{G_1(\A)}(\chi_1,s)$  if $m$ is odd and is a holomorphic section of $\Ind_{\til { P}_1(\A)}^{\til { G}_1(\A)}(\chi_1\chi_\psi,s)$ if $m$ is even.
\item 
   The section $R_S(g_1,f,s,\phi)$ is absolutely convergent for
  $\Re s > -(n-3)/2$ and can be meromorphically continued to the domain
  $\Re s > -(n-1)/2$.
\item
 The section $R_S(g_1,f,s,\phi)$ is holomorphic in $\Re s > - (n-2)/2$.
\item
 When $m$ is odd, the section $R_S(g_1,f,s,\phi)$ is holomorphic in  $\Re s > - (n-1)/2$ if $S$ is not in the square class of $(-1)^{(m (m-1)/2)}\Delta_Q$.
  \end{enumerate}
\end{proposition}
\begin{proof}
  This proposition was proved in \cite[Thm.~3.2, Lemma~3.3]{MR1295945} except for the last statement which is a refinement. See also \cite [Prop.~7.1] {MR1411571}. 

Assume that $m$ is odd. The unramified computation of \eqref{eq:R} in \cite [Pages 633--634] {MR1295945} leads to the factors
\begin{equation*}
  \frac{L^T (s+ \frac{n}{2},\chi \chi_S) }{L^T (2s+n)},
\end{equation*}
where $T$ denotes the set of `ramified' places. Note that the notation $m$ in \cite{MR1295945} is $1$.
Thus when $S$ is not in the square class of $(-1)^{(m (m-1)/2)}\Delta_Q$, there is no pole for $L^T (s+ \frac{n}{2},\chi \chi_S)$. 
\end{proof}

Now we will relate $R_S(g_1,f_{\Phi},s,\phi)$ to
$\Psi(g_1,\Phi,\phi)$. First we need a lemma to integrate the
`$z$-part'. We deviate from our usual notation and let $f_\Phi$ be a
weak Siegel-Weil section associated to $\Phi\in \cS(U(\A))$.

\begin{lemma}\label{lemma:SW_n=1}\quad
  Let $n=1$ and $S\in k^\times$. Assume $(m,r)\neq (3,1), (2,1)$ or
  $(m,r)=(3,1)$ and $S$ not in the square class of $-\Delta_Q$.  Let
  $w=\left(
    \begin{smallmatrix}
      0 & 1\\ -1 &0
    \end{smallmatrix}\right)$ and $s_0=\frac{m}{2}-1$. Then for $\Phi
  \in \cS(U(\A))$,
  \begin{equation}\label{eq:SW_n=1}
    \int_{\A}f_{\Phi}(w
    \begin{pmatrix}
      1& z\\ 0& 1
    \end{pmatrix},s)\overline{\psi_{S}(z/2)}dz
  \end{equation}
  can be meromorphically continued to the whole $s$-plane and is
  holomorphic at $s=s_0$. Its value at $s=s_0$ is $0$ if $\form{\ }{\ }_Q$ does not
  represent $S$. If $Q=\left(
    \begin{smallmatrix}
      S/2 & \\ & Q_1
    \end{smallmatrix}
  \right)$ then its value at $s=s_0$ is equal to the absolutely
  convergent integral
  \begin{equation}\label{eq:fourier-coeff-theta-integral}
    c_Q\int_{H_1(\A)\lmod H(\A)}\Phi(h^{-1}
    \begin{pmatrix}
      1\\ 0
    \end{pmatrix})dh,
  \end{equation}
where $c_Q = 1$ if $m \ge 5$ or $(m,r) = (4,0), (4,1), (3,0)$ and $c_Q=2$ if $(m,r) = (2,0)$.
\end{lemma}
\begin{proof}
   From proof of Prop.~\ref{prop:FJ-R}, \eqref{eq:SW_n=1} is holomorphic at $s=s_0$. (The modular character results in the shift of $s_0$.) 
Then by  \cite[Prop.~4.2]{MR887329} the two integrals are
  equal up to a constant which does not depend on $S$.
Note that \eqref{eq:SW_n=1} is the $S/2$-th Fourier coefficient $E_{S/2}(g,s,f_\Phi)$ of
  $ E(g,s,f_\Phi)$ while \eqref{eq:fourier-coeff-theta-integral} is the $S/2$-th Fourier coefficient of $c_Q I (g,\Phi)$ if the theta integral is absolutely convergent. When $m \ge 5$ or $(m,r) = (4,0), (4,1), (3,0), (2,0)$, the theta integral is absolutely convergent. Thus from the Siegel-Weil formula for $n=1$ we get that the value of $c_Q$ are as stated in the lemma.
\end{proof}
\begin{remark}\quad
 For the cases $(m,r)= (4,2)$ or $(3,1)$, $c_Q$ is not determined in this lemma. The values will be shown in Prop.~\ref {prop:nonsing_fourier_coeff} to be $1$.
  For the excluded case $(m,r)=(3,1)$ and $S$ in the square class of
  $-\Delta_Q$, $E_{S/2}(g,s,f_\Phi)$ has a pole at $s=s_0$. For the
  excluded case $(m,r)=(2,1)$, $E_{S/2}(g,s,f_\Phi)$ vanishes at
  $s=s_0$.
\end{remark}
\begin{proposition}\quad
  \label{prop:fj_eisenstein}
  Assume $m=n+1$. Also assume $(m,r)\neq (2,0), (2,1),(3,1) $ or
  $(m,r)=(3,1)$ and $S$ not in the square class of $-\Delta_Q$.  Let
  $\phi\in \cS_0(L(\A))$ and $f_{\Phi}(s)$ be the Siegel-Weil section
  of $I(\chi,s)$ associated to $\Phi\in \cS(U^n(\A))$.  If
  $\form{\cdot}{\cdot}_Q$ does not represent $S$ then
  $R(g_1,f_{\Phi},s,\phi)=0$.  If $\form{\cdot}{\cdot}_Q$  represents $S$, assume
  \begin{equation*}
    Q=
    \begin{pmatrix}
      S/2& \\ &Q_1
    \end{pmatrix}.
  \end{equation*}
  Then
  \begin{equation*}
    \FJ^\phi_S(g_1;E(s,f_\Phi))|_{s=0} = c_Q\int_{H_1(\A)\lmod H(\A)}
    E(g_1,s,f_{\Psi(\omega(h)\Phi,\phi)}) dh|_{s=0}
  \end{equation*}
\end{proposition}
\begin{proof}
  Embed $\Sp_2 (\A)$ into $G (\A) =\Sp_{2n} (\A)$ by
  \begin{equation*}
    g_0=
    \begin{pmatrix}
      a& b\\ c&d
    \end{pmatrix} \mapsto
    \begin{pmatrix}
      a &&b&\\
      &1_{n-1}&&0_{n-1}\\
      c &&d&\\
      &0_{n-1}&&1_{n-1}
    \end{pmatrix}
  \end{equation*} and denote this embedding by $\iota$. Also denote the
  lift $\til{\Sp}_2(\A) \rightarrow \til{\Sp}_{2n} (\A)$ by $\iota$.
  Observe that
  \begin{equation*}
    w_nv(0,0,z) = \iota(w
    \begin{pmatrix}
      1 & z \\ 0 &1
    \end{pmatrix}) w_{n-1}
  \end{equation*}
and the cocycles on both sides agree.
  To simplify notation temporarily set
  \begin{equation*}
    X =w_{n-1}
    \left(\begin{array}{c|c}
        1_n &
        \begin{matrix}
          0 & y\\ \trpz{y} & 0_{n-1}
        \end{matrix}
        \\\hline
        0_n & 1_n
      \end{array}\right) w_{n-1}g_1.
  \end{equation*}
   Then as a function of $g_0\in \til{\Sp}_2(\A)$,
  \begin{equation*}
    f_\Phi(\iota(g_0) X,s)
  \end{equation*}
  is a weak Siegel-Weil section associated to
  \begin{equation*}
    u \mapsto \omega(X)\Phi(u,0),
  \end{equation*}
  which is a Schwartz function in $\cS(U(\A))$.  Since we are simply restricting, the normalisation of $s$ is different from that in Lemma~\ref{lemma:SW_n=1}. Then by Lemma~\ref{lemma:SW_n=1},
 if $\form{\ }{\ }_Q$ does not represent $S$ then
  $R(g_1,f,0,\phi)=0$.
   If $Q=\left(
    \begin{smallmatrix}
      S/2 &\\ & Q_1
    \end{smallmatrix}\right)$ then using  Lemma~\ref{lemma:SW_n=1} to integrate the $z$-part,
   we find that $R(g_1,f_\Phi,s,\phi)|_{s=0}$ is equal
  to
  \begin{align*}
    c_Q\int_{L(\A)}\int_{H_1(\A)\lmod H(\A)} \omega(X)\Phi(h^{-1}
    \begin{pmatrix}
      1 & 0\\ 0&0
    \end{pmatrix})\overline{\omega_S(g_1)\phi(-y)}dhdy.
  \end{align*}
 The action of $\omega(X)$ can be explicitly computed:
  \begin{align*}
    &\omega(X)\Phi(h^{-1}
    \begin{pmatrix}
      1 & 0\\ 0&0
    \end{pmatrix})=\omega(g_1)\Phi(h^{-1}
    \begin{pmatrix}
      1 & -y\\ 0 &0
    \end{pmatrix}).
  \end{align*}
  Changing the order of integration and taking into account the
  definition of $\Psi$, we find
  \begin{equation*}
    R(g_1,f_\Phi,s,\phi)|_{s=0}=c_Q\int_{H_1(\A)\lmod H(\A)} \omega_{Q_1} (g_1)\Psi(0,\omega_Q(h)\Phi,\phi)dh,
  \end{equation*}
where we put back subscripts to indicate which Weil representation is used. This proves the proposition.
\end{proof}
\begin{remark}\quad
  The calculation relies on Lemma~\ref{lemma:SW_n=1}.  Thus for the
  case $(m,r)=(3,1)$ we can only go down to an Eisenstein series of
  lower rank for certain $S$'s. However for the proof of the Siegel-Weil
  formula in our case, this is  enough.
\end{remark}

\section{Proof of Siegel-Weil Formula}
\label{sec:pf_main_thm}

First we summarise  some results on irreducible submodules of the local
induced representations. In particular we note that they are nonsingular in the
sense of Howe\cite{MR777342}. See statement of
Lemma~\ref{lemma:nonsingular} for the precise definition.
We will fix one place
$v$ and let only $\Phi_v$ vary in the difference $A(g,\Phi)=
E(g,s,f_\Phi)|_{s=0} - 2 c_Q I_\REG(g,\Phi)$, where $c_Q$ is as in Lemma~\ref{lemma:SW_n=1}. It will be shown in the proof of Prop.~\ref{prop:A_B=0} that in fact those undetermined $c_Q$'s are all $1$.  We may interpret $A(g,\Phi)$
as a $\til{G} (k_v)$-intertwining operator from such an irreducible
nonsingular submodule to the space of automorphic forms. This helps deal with the $B$-th
Fourier coefficients of $A$ for $B$ not of full rank.

Fix $v$ a finite place of $k$ and suppress it from notation. We consider the local case.
Let $R_n(U)$ denote the image of the map
\begin{align*}
  \cS(U^n)&\rightarrow  I(\chi_\psi\chi,s_0)\\
  \Phi &\mapsto \omega(g)\Phi(0).
\end{align*}
This map induces an isomorphism $\cS(U^n)_H \isom R_n(U)$ by
\cite{MR743016}.  Let $U'$ be the quadratic space
with the same dimension and determinant as $U$ but with opposition
Hasse invariant. We form also $R_n(U')$. If no such $U'$ exists for reason of small dimension we just set $R_n(U')$ to $0$ . 
\begin{lemma} \quad
  Assume $m=n+1$. The $\til{G} (k_v)$-modules $R_n(U)$ and $R_n(U')$ are
  irreducible and unitarisable. Moreover $I(\chi_\psi\chi,s_0)\isom R_n(U)\oplus R_n(U')$,
\end{lemma}
We refer the reader to \cite [Cor.~3.7] {MR1194967} for the symplectic case and \cite [Cor.~4.14] {MR2843101} for the metaplectic case. 
A summary of decomposition of principal degenerate series is given in \cite [Prop.~7.2] {MR3166215}. 

\begin {lemma} \emph {\cite[Prop
    3.2(ii)]{MR961164}}
  \label{lemma:nonsingular}
   Assume $m= n+1$. Then $R_n(U)$ is a
    nonsingular representation of $\til{G} (k_v)$ in the sense of
    Howe\cite{MR777342}, i.e., there exists a
    Schwartz function $f$ on $\Sym_n(F_v)$ such that its Fourier
    transform $\hat{f}$ vanishes on all singular matrices in
    $\Sym_n(F_v)$ and such that the action of $f$ on $R_n(U)$ is not
    the zero action.
\end{lemma}

Combining the results above we are ready to show the Siegel-Weil
formula. Now we are back in the global case. Note the assumption that $m=n+1$.
Set $A(g,\Phi)= E(g,s,f_\Phi)|_{s=0} - 2c_Q I_\REG(g,\Phi)$, where $c_Q$ is as in Lemma~\ref{lemma:SW_n=1}.
\begin{proposition}\label{prop:nonsing_fourier_coeff}\quad
 Assume $m=n+1$. Assume that $U$ is not the split binary space. Then for $B
    \in \Sym_{n}(k)$ with rank $n$,
 the Fourier coefficients $ A_B=0$.
\end{proposition}\label{prop:A_B=0}
\begin{proof}
  The matrix $B$ is always congruent to  
  \begin{equation}\label{eq:B}
    \begin{pmatrix}
    S/2 & \\ & \beta
    \end{pmatrix}
  \end{equation}
for some $S\in k^\times$ and $\beta \in \Sym_{n-1} (k)$ nonsingular. We may just consider $B$-th Fourier coefficients for $B$ of the form \eqref{eq:B}. 
For $(m,r)= (3,1)$, we claim that we may further assume that $S$ is not in the square class of $-\Delta_Q$. In this case $B$ is a $2\times 2$-matrix. If $B=\left (\begin{smallmatrix}
    b_1 & \\ & b_2
    \end{smallmatrix}\right)$
with $b_1$ not  in the square class of $-2\Delta_Q$, then we are done already.
If $b_1 \equiv -2\Delta_Q$ and $b_2 \not\equiv -2\Delta_Q$, then 
\begin{equation*}
  \begin{pmatrix}
    0&1\\1&0
  \end{pmatrix}
\begin{pmatrix}
    b_1 & \\ & b_2
    \end{pmatrix}
\begin{pmatrix}
    0&1\\1&0
  \end{pmatrix}
\end{equation*}
has the desired property. Thus we just need to show that the quadratic form 
\begin{equation*}
  q (x) = \trpz { x} \begin{pmatrix}
      -2\Delta_Q & \\ & -2\Delta_Q
    \end{pmatrix} x
\end{equation*}
represents an element in $k^\times$ which is not in the square class of $-2\Delta_Q$. In other words we need to show that 
the quadratic form 
\begin{equation*}
  q (x) = \trpz { x} \begin{pmatrix}
      1 & \\ & 1
    \end{pmatrix} x
\end{equation*}
represents a non-square in $k^\times$. By \cite{MR0344216}, if we  take $a\in k^\times - ( k^\times)^2$ which is positive under each  embedding of $k$ into $\R$ and such that $(-1, a)_v =1$ for all $v\nmid \infty$, then $a$ is represented by $q$. Taking $a$ to be  a norm in $k [\sqrt {-1}]$ which is not a square, we have proved our claim.

Then we proceed to compute $B$-th Fourier coefficients for $B$ of the form \eqref{eq:B} and when $(m,r)= (3,1)$ we assume that $S$ is not in the square class of $-2\Delta_Q$.

 First we prove the anisotropic
  case $r=0$.   The base case $m=2$ and $n=1$ was proved in \cite[Chapter
  4]{MR887329}.  Now for $m=n+1$, we take $\psi_S$-th Fourier-Jacobi coefficient of $A (g,\Phi)$ and make use of Lemma~\ref{lemma:fj2fourier}. If $\form{\ }{\ }_Q$ does not
  represent $S$ then by Prop. \ref{prop:fj_theta} and
  Prop. \ref{prop:fj_eisenstein}  obviously $\FJ^\phi_{ S,\beta} (g_1, A (\cdot,\Phi)) = 0$ and hence $A_B=0$. If $\form{\ }{\ }_Q$
  represents $S$ then we can just assume that $Q=\left(
    \begin{smallmatrix}
      S/2 & \\ & Q_1
    \end{smallmatrix}\right)$. Note that $Q_1$ is still
  anisotropic.  Again by Prop. \ref{prop:fj_theta} and
  Prop. \ref{prop:fj_eisenstein} and the induction hypothesis we
  conclude that $A_B$=0.

  Secondly we assume $Q$ to be isotropic and $m\ge 5$, so $\form{\ }{\ }_Q$
  represents $S$. We can just assume that $Q=\left(
    \begin{smallmatrix}
      S/2 & \\ & Q_1
    \end{smallmatrix}\right)$.
  From Section \ref{sec:fj-coeff} we get by Prop. \ref{prop:fj_theta}
  and Prop. \ref{prop:fj_eisenstein} and the $m$ even case
  $\FJ_{S,\beta}^{\phi}(A)=0$ for all $\phi\in\cS(L(\A))$ if the rank of
  $\beta$ is $n-1$. Then by Lemma \ref{lemma:fj2fourier}, $A_B$
  vanishes for $B\in \Sym^n(k)$ such that $\det B \neq 0$.

The case where $m=4$ belongs to the even case and we know that the constant $c_Q$ must be $1$ for $m=4$ from the even case of Siegel-Weil formula.  Finally assume that $Q$ is isotropic and $m=3$. By the assumption on
  $S$ we can take $\psi_S$-th Fourier-Jacobi coefficients and still use Prop.~\ref{prop:fj_eisenstein} to do
  induction. Notice then $Q_1$ must be anisotropic and we are reduced
  to the case $(m,r)=(2,0)$. For now we have shown that for $(m,r)= (3,1)$, $ E(g,s,f_\Phi)|_{s=0}$ and $I_\REG(g,\Phi)$ are proportional. The constant of proportionality can be computed by comparing constant terms of both expressions. The constant term of $E(g,s,f_\Phi)|_{s=0}$ is given in \cite [Lemma~2.4] {MR946349}
  \begin{equation*}
    \sum_{k=0}^n \sum_{a\in Q_{n-k}\lmod \GL (n)} \int_{N_k' (\A)} f (w_{n-k}n' m (a)g,s )dn'.
  \end{equation*}
Here $Q_{n-k}$ is the parabolic subgroup of $\GL (n)$ whose Levi has blocks of sizes $k$ and $n-k$; $N_k'$ consists of unipotent elements of the form $n (b)$ with
\begin{equation*}
  b=
  \begin{pmatrix}
    0_k &0\\0&b_0
  \end{pmatrix}
\end{equation*}
where $b_0$ is of size $(n-k)\times (n-k)$; $w_{n-k}$ is the Weyl element
\begin{equation*}
  \begin{pmatrix}
    1_k & 0&0&0\\
    0&0&0&1_{n-k}\\
    0&0&1_k&0\\
    0&-1_{n-k}&0&0
  \end{pmatrix}.
\end{equation*}
Since the metaplectic group splits over $N (\A)$, the computation for symplectic group carries through. Denote the $k$-th term by $E_k (g,s,f)$ and restrict it to the subgroup $m (\GL (n,\A))$, so it becomes a function on $\GL (n,\A)$. We find that
\begin{equation*}
  E_k (( m (zI_n),\zeta)g,s,f) = \chi\chi_\psi (z^{2k-n},\zeta)|z|_\A^{(2k-n)s + n (n+1)/2 + k^2 - nk}E_k (g,s,f).
\end{equation*}
Thus $E_0$ and $E_n$ have the same central character and for $k\neq 0 ,n$, $E_k$  has a different central character. Note that $E_0 (g,s,f)=f (g,s)$, $E_n (g,s,f) = M (s)f (g,s)$ and the intertwining operator $M (0)$ acts as $1$ on the coherent part of $\Ind^{\til{G}(\A)}_{\til{P}(\A)} \chi\chi_\psi$ (c.f. \cite [Lemma~6.3] {MR3279536}), so in fact $E_0=E_n$. Next we consider the constant term of $I_\REG (g,\Phi)$. This was computed in Sec.~6 of \cite{MR961164}. We get
\begin{equation*}
  c_\alpha^{-1}\sum_{k=0}^{\min (r,n)} \sum_{\substack{x\in U^n (F) \\ \rank (x)=k, \form{x}{x}_Q = 0} }\int_{[H]} \omega ( g,h)\omega (\alpha)\Phi (x)dh.
\end{equation*}
Denote the $k$-th term by $I_k (g,\Phi)$. We have
\begin{equation*}
  I_k (( m (zI_n),\zeta)g,\Phi) = \chi\chi_\psi (z^n,\zeta) |z|_\A^{mn/2 - (m-k-1)k} I_k (g,\Phi).
\end{equation*}
For the case at hand with $(m,r)= (3,1)$ and $n=2$, there are only two terms $I_0$ and $I_1$. They have distinct `central characters'. Also note that in fact $I_0 (g,\Phi) = \omega (g,1)\Phi (0) = f_\Phi (g,0)$. When $s=0$, $I_0$, $E_0$ and $E_n$ have the same `central character'. From the equality $2I_0 = E_0 +E_n$, we conclude that
our $c_Q$ must be $1$.
\end{proof}
\begin{remark}\quad
  For the split binary case please refer to
  \cite{MR1289491} and note that the
  Eisenstein series vanishes at $0$, so the Siegel-Weil formula takes
  a different form.
\end{remark}
\begin{proof}[Proof of Theorem \ref{thm:Siegel-Weil}]
  Fix a finite place $v$ of $k$ and fix for each place $w$ not equal
  to $v$ a $\Phi_w^0\in \cS_0(U^n(k_w))$. Consider the map $A_v$ which
  sends $\Phi_v\in \cS_0(U^n (k_v))$ to $A(g,\Phi_v\otimes (\otimes_{w\neq
                                                                  v} \Phi_w^0))$. By invariant distribution theorem $R_n(U (k_v))\isom
  \cS(U^n (k_v))_{H_v}$. Thus $A_v$ defines a $\til{G} (k_v)$-intertwining operator
  \begin{equation*}
    \cS(U^n (k_v))\rightarrow \c{A}(G)
  \end{equation*}
  which actually factors through $R_n(U (k_v))$.

  As in Lemma
  \ref{lemma:nonsingular} we can find $f\in \cS(\Sym^n(k_v))$ such
  that its Fourier transform is supported on nonsingular symmetric
  matrices and $f$ does not act by zero. Then for all $g\in \til { G}(\A)$
  with $g_v=1$ and all $B\in \Sym_n(k)$ we have
  \begin{align*}
    (\rho(f).A(\Phi))_B(g)&= \int_{[\Sym_n]}
    \int_{\Sym_n(k_v)}f(c)A(\Phi) (ngn(c))\psi(-\tr(Bb))dc db\\
    &=\int_{[\Sym_n]}
    \int_{\Sym_n(k_v)}f(c)A(\Phi) (nn(c)g)\psi(-\tr(Bb))dc db\\
    &=\int_{[\Sym_n]}
    \int_{\Sym_n(k_v)}f(c)A(\Phi) (ng)\psi(-\tr(B(b-c)))dc db\\
    &=\hat{f}(B) A(\Phi)_B(g).
  \end{align*}
If $\rk B<n$,  the above is  $0$, since $\hat{f}(B)=0$. If $\rk
  B=n$, the above is again $0$, since by
  Prop.~\ref{prop:nonsing_fourier_coeff}, $A(\Phi)_B\equiv 0$. Thus $\rho(f) A(\Phi) =0$ as $G(k)\prod_{w\neq v} G ( k_w)$ is
  dense in $G(\A)$. Since $f$ does not act by zero and $R_n(U_v)$ is
  irreducible we find that in fact $A(\Phi)=0$.
This concludes the proof.
\end{proof}

\section{Rallis Inner Product Formula}
\label{sec:inner_product}
We will apply the regularised Siegel-Weil formula (Theorem
\ref{thm:Siegel-Weil}) to show the critical case of Rallis Inner Product formula
via the doubling method. We will also deduce the location of poles of
Langlands $L$-function from information on the theta lifting.

Let $G= \Sp_{2n}$ be the symplectic group of rank $n$. Let
$H=\r{O}(U,Q)$ with $(U,Q)$ a quadratic space of dimension $2n+1$ and
$\chi$ the character associated to $Q$ as in
\eqref{eq:chi-U-v}. Let $\pi$ be a genuine irreducible
cuspidal automorphic representation of $\til{G} (\A)$.  For $f\in\pi$
and $\Phi\in \cS(U^n(\A))$ define the theta lift of $f$ to $H$:
\begin{equation*}
  \Theta(h;f,\Phi)=\int_{[\til { G}]}f(g)\Theta(g,h;\Phi)dg.
\end{equation*}
To save space we will use the notation $[\til { G} ]$ to denote $G (k) \lmod   \til {G} (\A)$.
Note that as $f (g)$ and $\Theta(g,h;\Phi)$ are both genuine in $g$ the product can be viewed as a function on $G (\A)$.

Now consider $G$ to be the group of isometry of the
$2n$-dimensional space $W$ with symplectic form $\form {\ } {\ }_W$.  We will sometimes write $G (W)$ for the symplectic group. Let $W'$ be the symplectic space such that the underlying vector space is still $W$ and such that the symplectic form is $-1$ times  that of $W$. Then $W^\square=W\oplus W'$ is endowed with
the symplectic form $\form {\ } {\ }_{W^\square}$ such that
\begin{equation*}
  \form {(w_1, w_2)} {(w_1',w_2')}_{W^\square} = \form {w_1} {w_1'}_W + \form {w_2} {w_2'}_{W'} = \form {w_1} {w_1'}_W - \form {w_2} {w_2'}_W. 
\end{equation*}
Let $G^\square$ denote the `doubled' group $G (W^\square)$, which is the  symplectic group of rank $2n$. Let $\til{G}^\square (\A)$ be the corresponding metaplectic group.  If we fix a symplectic basis for $W$, the induced basis of $W'$ is not directly a symplectic basis. The isomorphism between the two symplectic groups $G (W)$ and $G (W')$ is given by the conjugation by an element in the similitude symplectic group with similitude $-1$.  For $g\in G (\A)$ set
\begin{equation*}
  \nu (g) =
  \begin{pmatrix}
    1_n &\\ & -1_n
  \end{pmatrix}
  g\begin{pmatrix}
    1_n &\\ & -1_n
  \end{pmatrix}.
\end{equation*}
The conjugation $\nu$ extends uniquely to an automorphism of $\til{G} (\A)$
\begin{equation*}
  \nu ( (g,\zeta)) = (\nu (g), \epsilon (g)\zeta)
\end{equation*}
where $\epsilon (g) =\pm 1$. More precisely, it is determined as follows:
\begin{align*}
  \epsilon (m (A)) &= \chi_{-1} (\det A)\\
  \epsilon (n (B)) &= 1\\
  \epsilon (w_n^{-1}) &=1.
\end{align*}
We have the natural embedding 
\begin{align*}
  \iota: G  (\A) \times G  (\A) &\rightarrow G^\square (\A) \\
  ( g_1, g_2) &\mapsto \begin{pmatrix}
    a_1 & & b_1&\\
    &a_2&&-b_2\\
    c_1&&d_1& \\
    & -c_2&&d_2
  \end{pmatrix}
\end{align*}
meaning that the first copy of $G (\A)$ is mapped to the subgroup $G (W) (\A)$ of $ G^\square (\A)$ and the second copy is mapped to the subgroup $G (W') (\A)$ with identification given as above.
The embedding $\iota$  lifts to a homomorphism
\begin{align*}
  \til {\iota}: \til{G}(\A)\times \til{G}(\A)&\rightarrow \til{G}^\square(\A)\\
((g_1,\zeta_1),(g_2,\zeta_2))&\mapsto (\begin{pmatrix}
    a_1 & & b_1&\\
    &a_2&&-b_2\\
    c_1&&d_1& \\
    & -c_2&&d_2
  \end{pmatrix},\epsilon (g_2)\zeta_1\zeta_2).
\end{align*}
With this we find $\Theta(\til {\iota}(g_1,g_2),h;\Phi)=
\Theta(g_1,h;\Phi_1)\overline{\Theta(g_2,h;\Phi_2)}$ for  $g_i \in \til {G} (\A)$ if we set
$\Phi=\Phi_1\otimes \overline{\Phi}_2$ for $\Phi_i\in \cS(U^n(\A))$.

Consider the inner product between two theta lifts. Let $f_i\in \pi$ and $\Phi_i \in \cS(U^n(\A))$. 
Suppose the inner product
\begin{equation*}
  \langle \Theta(f_1,\Phi_1),
  \Theta(f_2,\Phi_2)\rangle=\int_{[H]}
  \int_{ [\til { G} \times \til {G} ]}
  f_1(g_1)\Theta(g_1,h;\Phi_1) 
  \overline{f_2(g_2)\Theta(g_2,h;\Phi_2)} dg_1dg_2dh
\end{equation*}
is absolutely convergent. Then we exchange order of integration to get
\begin{align*}
  &\int_{ [\til { G} \times \til {G}]}
  f_1(g_1)\overline{f_2(g_2)} \left( \int_{[H]}\Theta(g_1,h;\Phi_1)
    \overline{\Theta(g_2,h;\Phi_2)}dh \right)dg_1
  dg_2\\
  =&\int_{[\til { G} \times \til {G}]}  f_1(g_1)\overline{f_2(g_2)}\left( \int_{[H]}\Theta(\til {\iota}(g_1,g_2),h;\Phi)\right)dg_1 dg_2\\
  =&\int_{[\til { G} \times \til {G}]}
  f_1(g_1)\overline{f_2(g_2)} I(\til {\iota}(g_1,g_2);\Phi) dg_1 dg_2.
\end{align*}
Thus we define the regularised inner product by
\begin{equation}
  \label{eq:reg_inner_product}
  \langle \Theta(f_1,\Phi_1),
  \Theta(f_2,\Phi_2)\rangle_\REG
  =\int_{[\til { G} \times \til {G}]}
  f_1(g_1)\overline{f_2(g_2)}I_{\REG}(\til {\iota}(g_1,g_2);\Phi)dg_1 dg_2
\end{equation}
in case the usual inner product diverges.

Let $W=X\oplus Y$ be a polarisation of $W$. Then the
Weil representation $\omega$ considered up till now is in fact
realised on $\cS(U\otimes X(\A))$. Thus $\Phi$ lies in $\cS(U\otimes ( X\oplus X)(\A))$.
We could apply the Siegel-Weil formula now, but then we would not be
able to use the basic identity in
\cite{MR892097} directly. Thus we proceed as follows. The space $U\otimes W^\square$ has two
complete polarisations
\begin{align*}
  U\otimes W^\square&=(U\otimes (X\oplus
X))\oplus(U\otimes (Y\oplus Y));\\
U\otimes W^\square&=(U\otimes
 W^\Delta)\oplus (U\otimes W^\nabla)
\end{align*}
 where $W^\Delta=\{(w,w)|w\in W\}$ and
$W^\nabla=\{(w,-w)|w\in W\}$. Let $P^\square$ be the Siegel parabolic of
$G^\square$ fixing the maximal isotropic subspace $W^\nabla$. There is an
isometry given by Fourier transform
\begin{equation*}
  \delta: \cS((U\otimes (X\oplus X))(\A)) \rightarrow \cS((U\otimes W^\Delta)(\A))
\end{equation*}
intertwining the action of $\til{G}^\square(\A)$.

Then \eqref{eq:reg_inner_product} is equal to
\begin{equation*}
  \int_{[\til {G}\times \til {G}]}
  f_1(g_1)\overline{f_2(g_2)}I_{\REG}(\til {\iota}(g_1,g_2);\delta\Phi)dg_1 dg_2.
\end{equation*}
Note that here the theta function implicit in $I_{\REG}$ is associated to
the Weil representation realised on $\cS(U\otimes W^\Delta(\A))$. Now we
apply the regularised Siegel-Weil formula to get
\begin{equation*}
  2^{-1}\int_{[\til {G}\times \til {G}]}
  f_1(g_1)\overline{f_2(g_2)}E(\til {\iota}(g_1,g_2),s,F_{\delta\Phi})|_{s=0}dg_1 dg_2
\end{equation*}
where to avoid conflict of notation we use $F_{\delta\Phi}$ to denote
the Siegel-Weil section in
$\Ind_{\til{P}^\square(\A)}^{\til{G}^\square(\A)}\chi_\psi\chi|\ |^s$ associated to
$\delta\Phi$.

For  a section $F$ in this induced representation, 
set the zeta function to be
\begin{equation}\label{eq:def_zeta_function}
  Z(f_1,f_2,s,F)=\int_{[\til {G}\times \til {G}]}
  f_1(g_1)\overline{f_2(g_2)}E(\til {\iota}(g_1,g_2),s,F)dg_1 dg_2.
\end{equation}
Thus
\begin{equation}\label{eq:theta-pair=zeta-function}
  \langle \Theta(f_1,\Phi_1),
  \Theta(f_2,\Phi_2)\rangle_\REG = 2^{-1} Z(f_1,f_2,0,F).
\end{equation}
By the basic identity in
\cite{MR892097} generalised to the
metaplectic case and by
\cite[Eq. (25)]{MR1166512}, the zeta function $Z(f_1,f_2,s,F)$ is
equal to
\begin{align*}
  &\int_{\til{G} (\A)}F( \til\iota(g,1),s)
  \int_{[\til {G}]} f_1(g'g)\overline{f_2(g')}dg' dg\\
  &=\int_{\til{G} (\A)}F( \til\iota(g,1),s).
  \langle  \pi(g)f_1,f_2\rangle dg.
\end{align*}
In fact the integrand is a function of $G (\A)$.
Suppose $F$ and $f_i$ are factorisable. Then the above factorises into
a product of local zeta integrals
\begin{equation*}
  Z(f_{1,v},f_{2,v},s,F_v)=\int_{\til { G} (k_v)}F_v( \iota(g_v,1),s)
  \langle\pi_v(g_v)f_{1,v},f_{2,v}\rangle dg_v.
\end{equation*}

Let $S$ be a finite set of places of $k$ containing all the
archimedean places, even places, outside which $\pi_v$ is an
unramified principal series representation, $f_{i,v}$ is $K_{G,v}$-invariant and
normalised, $F_v$ is a normalised spherical section and
$\psi_v$ is unramified. 
Then by \cite[Prop. 4.6]{MR1166512} the
local integral $Z(f_{1,v},f_{2,v},s,F_v)$ is equal to
\begin{equation*}
  \frac{L_v(s+\frac{1}{2},\pi_v \times \chi_v,\psi_v)}{\til{d}_{G^\square_v}(s)} 
\end{equation*}
where $L_v(s+\frac{1}{2},\pi_v \times \chi_v,\psi_v)$ is the Langlands
$L$-function as developed in \cite{gan:_repres_metap_group_i} and where 
\begin{equation*}
  \til{d}_{G^\square_v}(s)= \prod_{i=1}^{n}\zeta_v(2s+2i).
\end{equation*}
Note here we normalise the Haar measure on $\til{G }_v$ so that
$\til { K}_{G_v}$ has volume $1$. Our $\til{d}_{G^\square_v}$ is slightly different from that in \cite {MR1166512}, because we are using a different local $L$-factor. Compare with the statement in \cite [Prop.~6.1] {MR3006697}. Note also that the definition of this
$L$-function depends on the chosen additive character $\psi_v$.

For $v \in S$, the local $L$-factor is shown to be the `g.c.d' of the zeta integrals of `good' sections by Yamana in \cite{MR3211043}. Since we are using holomorphic sections, we need to account for the difference by having $\til{d}_{G^\square_v}(s)$ in our formula. Set
\begin{equation*}
  \cZ (f_{1,v},f_{2,v},s,F_v) =  Z(f_{1,v},f_{2,v},s,F_v) (L_v(s+\frac{1}{2},\pi_v \times \chi_v,\psi_v))^{-1} \til{d}_{G^\square_v}(s).
\end{equation*}
Then $\cZ (f_{1,v},f_{2,v},s,F_v)$ is entire and for all $s\in \CC$ and all places $v$, there exist $f_{1,v}, f_{2,v} \in \pi$ and $F_v \in \Ind_{\til{P}^\square(\A)}^{\til{G}^\square(\A)}\chi_\psi\chi|\ |^s$ such that it is non-vanishing. We have that
\begin{equation*}\label{eq:eis=unram*ram}
  Z (f_1,f_2,s,F) = \frac { L(s+\frac{1}{2},\pi \times \chi,\psi)} {\til{d}_{G^\square}(s)}\prod_{v\in S} \cZ (f_{1,v},f_{2,v},s,F_v).
\end{equation*}

We set $s$ to $0$ in the zeta function to get the Rallis
inner product formula.  
\begin{theorem}\label{thm:rallis_inner_product}\quad
  Suppose $m=2n+1$. Then
  \begin{equation*}
    \langle \Theta(f_1,\Phi_1),
    \Theta(f_2,\Phi_2)\rangle_\REG = 
    \frac{2^{-1} L(\frac{1}{2},\pi \times \chi,\psi)}{\til{d}_{G^\square}(0)} \cdot \prod_{v\in S} \cZ (f_{1,v},f_{2,v},0,F_v),
\end {equation*}
where $F$ is the Siegel-Weil section in $\Ind_{\til{P}^\square(\A)}^{\til{G}^\square(\A)}\chi_\psi\chi|\ |^s$ associated to the Schwartz function $\delta (\Phi_1\otimes \overline {\Phi_2})$.
\end{theorem}
Returning to \eqref{eq:eis=unram*ram} we deduce the possible location of poles of $L$-function.
\begin{proposition}\quad
  The poles of $L(s+\frac{1}{2},\pi \times \chi,\psi)$ in $\Re(s)\ge 1/2$ are simple and
  are contained in the set
  \begin{equation*}
    \{\frac{3}{2},\frac{5}{2},\cdots, n+\frac{1}{2}\}.
  \end{equation*}
\end{proposition}
\begin{proof}
 We note that the poles of
  $L(s+\frac{1}{2},\pi \times \chi,\psi)$ are contained in the set of
  poles of $\til{d}_{G_{2}}(s)E(s,\iota(g_1,g_2),F)$. The poles of
  the Eisenstein series in $\Re(s)\ge 0$ are simple and are contained in
  $\{1,2,\ldots,n\}$, c.f. \cite[Page
  216]{MR1863861}. From this we get the proposition.
\end{proof}

We give an analogue of Kudla and Rallis's
\cite[Thm. 7.2.5]{MR1289491} in the case where $m=2n+1$. 
Let $\Theta^{U}(\pi)$ denote the space of theta lift of $\pi$ to the orthogonal group of $U$. Recall that we have defined the sub-modules $R_n(U_v)$ of $\Ind_{\til{P}^\square(k_v)}^{\til{G}^\square(k_v)}\chi_{\psi,v} \chi_v$ for each place $v$ to be the image of the map
\begin{align*}
  \cS(U (k_v)^n)&\rightarrow  \Ind_{\til{P}^\square(k_v)}^{\til{G}^\square(k_v)}\chi_{\psi,v} \chi_v\\
  \Phi &\mapsto \omega(g)\Phi(0).
\end{align*}
The parameter $s$ here is $0$ since we require the dimension $m$ of $U$ to be equal to $2n+1$.
The structure of $\Ind_{\til{P}^\square(k_v)}^{\til{G}^\square(k_v)}\chi_{\psi,v} \chi_v$  is summarised by \cite [Propositions~5.2, 5.3] {MR3279536}:
\begin{proposition}\quad
 We have
 \begin{equation*}
   \Ind_{\til{P}^\square(k_v)}^{\til{G}^\square(k_v)}\chi_{\psi,v} \chi_v = \bigoplus_{U_v: \chi_{U_v} = \chi} R_n(U_v).
 \end{equation*}
\end{proposition}
\begin{theorem}\label{thm:nonvanishing_theta_lifts}\quad Let $\chi$ be a quadratic character.
Suppose
    $L(s,\pi \times \chi,\psi)$ does not vanish at $s=1/2$.
    Then there exists a quadratic space $U$ over $k$ with dimension
    $m$ and $\chi_{U}=\chi$ such that $\Theta^{U}(\pi) \neq 0$.
 \end{theorem}
\begin{proof}
For $v\in S$ we choose $f_{1,v}, f_{2,v}\in \pi$ and $F_v\in \Ind_{\til{P}^\square(\A)}^{\til{G}^\square(\A)}\chi_\psi\chi|\ |^s$ such that $\cZ (f_{1,v},f_{2,v},s,F_v)$ is non-vanishing at $s=0$. Furthermore because of the module structure of $\Ind_{\til{P}^\square(k_v)}^{\til{G}^\square(k_v)}\chi_{\psi,v} \chi_v$, we may take $F_v$ to belong to some $R_n (U_v)$ with $\chi_{U_v}=\chi_v$. By \cite [Prop.~6.2] {MR3279536}, the Eisenstein series of an incoherent section of $\Ind_{\til{P}^\square(\A)}^{\til{G}^\square(\A)}\chi_\psi\chi|\ |^s$ vanishes at $s=0$. In this way we may assume that $F$ belongs to $\prod_v R_n (U_v)$ for   a global quadratic space $U$.

  As $L(s + \frac {1} {2},\pi \times \chi,\psi)$ is assumed to be non-vanishing at $s=0$, $\til{d}_{G^\square}(s)$ does not have a pole at $s=0$ and the data are chosen so that $\cZ (f_{1,v},f_{2,v},s,F_v)$ is non-vanishing at $s=0$,  the regularised pairing of theta lifts $\langle \Theta(f_1,\Phi_1),
    \Theta(f_2,\Phi_2)\rangle_\REG$ is non-vanishing. This means that $\Theta^{U}(\pi) \neq 0$ for this $U$.
\end{proof}

\section{Conclusion}
\label{sec:conclusion}

The Siegel-Weil formula has been widely used in the literature and much work has been done toward its proof in various cases. In this article we have written down a complete proof for the regularised Siegel-Weil formula in the boundary case and thus closed the gap in literature. From it we have deduced  the Rallis inner product formula in the critical case. As this case of the Rallis inner product formula expresses the central value of the L-function $L (s,\pi\times\chi,\psi)$ in terms of the inner product of theta lifts, it is of particular importance. It also forms the foundation of the arithmetic Rallis inner product formula which relates the central derivative of the L-function to the conjectured Beilinson-Bloch height pairing of arithmetic theta lifts. This will be part of our future research project.


\end{document}